\documentclass[11pt]{article}

\usepackage{hyperref}

\usepackage{amsmath}
\usepackage{amssymb}
\usepackage{amsthm}
\usepackage{amscd}
\usepackage{amsfonts}
\usepackage[all]{xy}
\usepackage{ifthen}
\usepackage{a4}

\long\def\forget#1{}

\newdir^{ (}{{}*!/-3pt/\dir^{(}}    
\newdir^{  }{{}*!/-3pt/\dir^{}}    
\newdir_{ (}{{}*!/-3pt/\dir_{(}}    
\newdir_{  }{{}*!/-3pt/\dir_{}}

\def\quotes#1{{''#1''}} 

\def\INDENT{\hspace*{\parindent}}

\catcode`\ä=\active
\catcode`\ö=\active
\catcode`\ü=\active
\catcode`\Ä=\active
\catcode`\Ö=\active
\catcode`\Ü=\active
\catcode`\ß=\active
\catcode`\é=\active
\catcode`\è=\active
\catcode`\É=\active

\defä{\"a}
\defö{\"o}
\defü{\"u}
\defÄ{\"A}
\defÖ{\"O}
\defÜ{\"U}
\letß=\ss
\defé{\'{e}}
\defè{\`{e}}
\defÉ{\'{E}}

\DeclareMathOperator{\Quot}{Quot}
\DeclareMathOperator{\Spec}{Spec}

\DeclareMathOperator{\id}{id}
\DeclareMathOperator{\rank}{rank}
\DeclareMathOperator{\coker}{coker}
\DeclareMathOperator{\im}{im}

\newcommand{\sep}{{\rm sep}}
\newcommand{\alg}{{\rm alg}}

\DeclareMathOperator{\weight}{wt}
\DeclareMathOperator{\supp}{supp}

\DeclareMathOperator{\Hom}{Hom}
\DeclareMathOperator{\End}{End}
\DeclareMathOperator{\Isog}{Isog}
\DeclareMathOperator{\QHom}{QHom}
\DeclareMathOperator{\QEnd}{QEnd}
\DeclareMathOperator{\QIsog}{QIsog}
\DeclareMathOperator{\Gal}{Gal}

\newtheoremstyle{statement}%
{2\parskip}
{\parskip}
{\itshape}
{}
{\bfseries}
{}
{\newline}
{}

\newtheoremstyle{notice}%
{2\parskip}
{\parskip}
{}
{}
{\itshape}
{}
{\newline}
{}

\newtheorem{all}{all}[section]

\theoremstyle{plain}
\newtheorem{Definition}[all]{Definition}
\newtheorem{Lemma}[all]{Lemma}
\newtheorem{Proposition}[all]{Proposition}
\newtheorem{Corollary}[all]{Corollary}
\newtheorem{Theorem}[all]{Theorem}

\theoremstyle{remark}

\theoremstyle{definition}
\newtheorem{Remark}[all]{Remark}
\newtheorem{Example}[all]{Example}

\newenvironment{suchthat}
{\setlength{\parskip}{0ex}%
\begin{enumerate}\setlength{\parskip}{0ex}\setlength{\itemsep}{0ex}%
}{%
\end{enumerate}%
}%

\def\mal{^{\SSC\times}}

\def\II#1{{[\,#1\,]}}  
\def\dual#1{{#1}^\vee} 

\def\Z{\mathbb{Z}} 
\def\N{\mathbb{N}} 
\def\Q{\mathbb{Q}} 
\def\Ff{{\mathbb{F}}} 
\def\Fq{{\mathbb{F}_q}} 
\def\FqItI{{\Fq\II{t}}} 
\def\AA{\mathbb{A}} 
\def\PP{\mathbb{P}} 

\def\O{{\cal O}} 
\def\I{{\cal J}} 
\def\F{{\cal F}} 
\def\G{{\cal G}} 
\def\M{{\cal M}} 
\def\HOM{{\cal H}\mbox{\it om}} 

\def\OC{{\O_C}} 
\def\OS{{\O_S}}
\def\CS{{C_S}}
\def\OCS{{\O_\CS}}
\def\CL{{C_L}}
\def\OCL{{\O_\CL}}

\def\FF{{\underline{\F}}} 
\def\GG{{\underline{\G}}} 
\def\ZZ{{\underline{0}}} 

\def\ulM{{\underline{M\!}\,}{}}
\def\ulN{{\underline{N\!}\,}{}}
\def\ulTM{{\underline{\,\,\wt{\!\!M}\!}\,}{}}
\def\ulTN{{\underline{\wt{N}\!}\,}{}}
\def\ulHM{{\underline{\hat M\!}\,}{}}
\def\ulHN{{\underline{\hat N\!}\,}{}}

\renewcommand{\mod}{\operatorname{mod}}

\def\P{{\mit\Pi}} 
\def\t{{\mit\tau}} 
\def\s{{\sigma^\ast}} 

\def\TP{\widetilde{\P}} 
\def\Tt{\tilde{\tau}}
\def\TF{\widetilde{\F}}
\def\TFF{\widetilde{\FF}}

\def\Tr{\tilde{r}}
\def\Td{\tilde{d}}

\def\HP{\widehat{\P}} 
\def\Ht{\hat{\tau}}
\def\HF{\widehat{\F}}

\def\Hr{\hat{r}}

\def\chr{\varepsilon} 
\def\otimesidOCL#1{\!\otimes1} 

\def\matr#1#2#3#4{\left(\genfrac{}{}{0pt}{}{#1}{#2}\,\genfrac{}{}{0pt}{}{#3}{#4}\right)}

\def\vect#1#2{\left(\genfrac{}{}{0pt}{}{#1}{#2}\right)}
\def\tvect#1#2{{\textstyle\big(\genfrac{}{}{0pt}{}{#1}{#2}\big)}}

\def\TA{\tilde{A}}
\def\Lsep{{L^{\sep}}}
\def\liminv#1{\displaystyle\lim_{\stackrel{\textstyle\longleftarrow}{#1}}}
\def\CLa{{C'_L}} 
\def\Av{{A_v}} 
\def\AvG{{\Av\II{G}}} 
\def\Qv{{Q_v}} 
\def\QvG{{\Qv\II{G}}} 
\def\AxL{{\tilde A\otimes_\Fq L}} 
\def\VvFF{{V_v\FF}} 

\def\smallexact#1#2#3#4#5#6#7#8#9{%
{\,#1\rightarrow#3\rightarrow#5%
\ifthenelse{\equal{#7}{}}{}{\rightarrow#7%
\ifthenelse{\equal{#9}{}}{}{\rightarrow#9}}%
\,}}

\def\exact#1#2#3#4#5#6#7#8#9{%
{#1\longrightarrow#3\longrightarrow#5%
\ifthenelse{\equal{#7}{}}{}{\longrightarrow#7%
\ifthenelse{\equal{#9}{}}{}{\longrightarrow#9}}%
\,}}

\def\bigexact#1#2#3#4#5#6#7#8#9{%
\begin{CD}%
#1 @>{#2}>> #3 @>{#4}>> #5 
\ifthenelse{\equal{#7}{}}{}{@>{#6}>> #7 
\ifthenelse{\equal{#9}{}}{}{@>{#8}>> #9 }}
\end{CD}%
}

\newcommand{\DS}{\displaystyle}
\newcommand{\TS}{\textstyle}
\newcommand{\SC}{\scriptstyle}
\newcommand{\SSC}{\scriptscriptstyle}
\let\setminus\smallsetminus
\DeclareMathOperator{\GL}{GL}
\DeclareMathOperator{\Graph}{Graph}

\def\isoto{\stackrel{}{\mbox{\hspace{1mm}\raisebox{+1.4mm}{$\SC\sim$}\hspace{-3.5mm}$\longrightarrow$}}}
\newcommand{\longto}{\longrightarrow}

\newcommand{\shortonto}{\mbox{\mathsurround=0pt \;$\to \hspace{-0.8em} \to$\;}}
\newcommand{\into}{ \mbox{\mathsurround=0pt \;\raisebox{0.63ex}{\small $\subset$} \hspace{-1.07em} $\longrightarrow$\;}}
\newcommand{\es}{\enspace}
\newcommand{\et}{{\rm \acute{e}t}}
\newcommand{\nil}{{\rm nil}}
\DeclareMathOperator{\rk}{rk}
\newcommand{\wt}{\widetilde}
\newcommand{\wh}{\widehat}
\newcommand{\Fa}{{\mathfrak{a}}}
\newcommand{\ulK}{{\ul K}}
\newcommand{\BF}{{\mathbb{F}}}
\newcommand{\CM}{{\cal{M}}}
\newcommand{\ul}[1]{{\underline{#1}}}
\newcommand{\dbl}{{\mathchoice{\mbox{\rm [\hspace{-0.15em}[}}
                              {\mbox{\rm [\hspace{-0.15em}[}}
                              {\mbox{\scriptsize\rm [\hspace{-0.15em}[}}
                              {\mbox{\tiny\rm [\hspace{-0.15em}[}}}}
\newcommand{\dbr}{{\mathchoice{\mbox{\rm ]\hspace{-0.15em}]}}
                              {\mbox{\rm ]\hspace{-0.15em}]}}
                              {\mbox{\scriptsize\rm ]\hspace{-0.15em}]}}
                              {\mbox{\tiny\rm ]\hspace{-0.15em}]}}}}
\newcommand{\dpl}{{\mathchoice{\mbox{\rm (\hspace{-0.15em}(}}
                              {\mbox{\rm (\hspace{-0.15em}(}}
                              {\mbox{\scriptsize\rm (\hspace{-0.15em}(}}
                              {\mbox{\tiny\rm (\hspace{-0.15em}(}}}}
\newcommand{\dpr}{{\mathchoice{\mbox{\rm )\hspace{-0.15em})}}
                              {\mbox{\rm )\hspace{-0.15em})}}
                              {\mbox{\scriptsize\rm )\hspace{-0.15em})}}
                              {\mbox{\tiny\rm )\hspace{-0.15em})}}}}

\def\?{\ 
???\ \immediate\write16{}
\immediate\write16{Warning: There was still a question mark . . . }
\immediate\write16{}}

\newcommand{\BHBPropAAA}{7.3} 
\newcommand{\BHBThmBBB}{6.11}  

\begin{document}

\author{Matthias Bornhofen, Urs Hartl%
\footnote{Corresponding author: Urs Hartl, Institute of Mathematics, University of Muenster, Einsteinstr.\ 62, D--48149 Muenster, Germany, http:/\!/www.math.uni-muenster.de/u/urs.hartl/ }}

\title{Pure Anderson Motives and Abelian $\tau$-Sheaves}


\maketitle

\begin{abstract}
\noindent
Pure $t$-motives were introduced by G. Anderson as higher dimensional generalizations of Drinfeld modules, and as the appropriate analogs of abelian varieties in the arithmetic of function fields. In order to construct moduli spaces for pure $t$-motives the second author has previously introduced the concept of abelian $\tau$-sheaf. In this article we clarify the relation between pure $t$-motives and abelian $\tau$-sheaves. We obtain an equivalence of the respective quasi-isogeny categories.
 Furthermore, we develop the elementary theory of both structures regarding morphisms, isogenies, Tate modules, and local shtukas. The later are the analogs of $p$-divisible groups.

\end{abstract}

\tableofcontents


\thispagestyle{empty}

\setcounter{section}{-1}
\section{Introduction}

Important objects in the arithmetic of number fields are elliptic curves and abelian varieties. Their theory has been vastly developed in the last two centuries. For the arithmetic of function fields Drinfeld \cite{Drinfeld,Drinfeld3} has invented the concepts of \emph{elliptic modules} (today called \emph{Drinfeld modules}) and \emph{elliptic sheaves} in the 1970's, both as the analogs of elliptic curves. Since then, the arithmetic of function fields has evolved into an equally rich parallel world to the arithmetic of number fields.
As for higher dimensional generalizations of elliptic modules or sheaves there are different notions, for instance Anderson's \emph{abelian $t$-modules} and \emph{$t$-motives} \cite{Anderson}, \emph{Drinfeld-Anderson shtukas} \cite{Drinfeld5}, or \emph{abelian $\tau$-sheaves} which were introduced by the second author in \cite{Hl} in order to construct moduli spaces for pure $t$-motives. The generalization of (pure) $t$-motives to (pure) $A$-motives, already immanent in Anderson's work was elaborated in \cite{Heiden}. In the present article we advertise the point of view that pure $A$-motives (which we also call \emph{pure Anderson motives}) and abelian $\tau$-sheaves are the appropriate analogs for abelian varieties. This is also supported by the results in \cite{Hl} and \cite{BH_B}. It is due to the fact that both structures have the feature of purity built in as opposed to general $t$-motives or Drinfeld-Anderson shtukas. For example non-zero morphisms exist only between pure $A$-motives or abelian $\tau$-sheaves of the same weight (see \ref{PROP.1} and \ref{Cor2.9b} in the body of the article).

There is a strong relation between pure $A$-motives, and abelian $\tau$-sheaves. To give their definition let $C$ be a connected smooth projective curve over $\Ff_q$, let $\infty\in C(\Ff_q)$ be a fixed point, and let $A=\Gamma(C\setminus\{\infty\},\O_C)$. For a field extension $L\supset\Fq$ let $\s$ be the endomorphism of $A_L:=A\otimes_\Fq L$ sending $a\otimes b$ to $a\otimes b^q$ for $a\in A$ and $b\in L$. Let $c^\ast:A\to L$ be an $\Fq$-homomorphism and let $J=(a\otimes 1-1\otimes c^\ast(a):a\in A)\subset A_L$. A \emph{pure $A$-motive $\ulM=(M,\tau)$ of rank $r$, dimension $d$ and characteristic $c^\ast$} consists of a locally free $A_L$-module $M$ of rank $r$ and an $A_L$-homomorphism $\t:\s M:= M\otimes_{A_L,\s}A_L\to M$ with $\dim_L\coker\t=d$ and $J^d\cdot\coker\t=0$, such that $M$ possesses an extension to a locally free sheaf $\CM$ on $C_L:=C\times_\Fq L$ on which $\t^l:(\s)^l\CM\to\CM(k\cdot\infty)$ is an isomorphism near $\infty$ for some positive integers $k$ and $l$. The last condition is the purity condition. The ratio $\frac{k}{l}$ equals $\frac{d}{r}$ (see \ref{Prop1'.1b}) and is called the \emph{weight of $\ulM$}. Anderson's definition of pure $t$-motives~\cite{Anderson} is recovered by setting $C=\PP^1_\Fq$ and $A=\Fq[t]$. 

In addition to this data an \emph{abelian $\tau$-sheaf} consists of a sequence of sheaves $\CM_i\subsetneq\CM_{i+1}$ lying between $\CM_0:=\CM$ and $\CM_l:=\CM(k\cdot\infty)$ whose stalks at $\infty$ are the images of $\t^i$ for $i=0,\ldots,l$ (see \ref{Def1.1}). The quasi-isogeny categories of pure $A$-motives and abelian $\tau$-sheaves are equivalent (\ref{PropX.1}, \ref{Cor2.9d}). An abelian $\tau$-sheaf of dimension $d=1$ is the same as an elliptic sheaf. In this sense abelian $\tau$-sheaves are higher dimensional elliptic sheaves. 
The concept of abelian $\tau$-sheaves was introduced by the second author~\cite{Hl} for the following reasons. In contrast to pure $A$-motives, abelian $\tau$-sheaves possess nice moduli spaces which are Deligne-Mumford stacks locally of finite type and separated over $C$; see~\cite{Hl}. Moreover, let $c:\Spec L\to \Spec A\subset C$ be the morphism induced by $c^\ast$. The notion of abelian $\tau$-sheaves is still meaningful if $c:\Spec L\to C$ is not required to factor through $\Spec A$. Indeed, the possibility to have $\im(c)=\infty$ was crucial for the uniformization of the moduli spaces of abelian $\tau$-sheaves and the derived results on analytic uniformization of pure $A$-motives in \cite{Hl}. For these reasons we develop the theory of abelian $\tau$-sheaves and pure $A$-motives simultaneously in the present article.

Let $Q$ be the function field of $C$. Then the endomorphism algebra of a pure $A$-motive or an abelian $\tau$-sheaf is a finite dimensional $Q$-algebra (\ref{PropT.2}, \ref{ThmT.3}). 
In contrast the endomorphism algebra of an abelian variety is a finite dimensional algebra over the rational numbers. Through this fact pure $A$-motives and abelian $\tau$-sheaves belong to the arithmetic of function fields. 
We further investigate their \mbox{(quasi-)}isogenies. An isogeny $f:(M,\tau)\to(M',\tau')$ between pure $A$-motives of the same characteristic is an injective morphism $f:M\to M'$ with torsion cokernel such that $f\circ\tau=\tau'\circ\s f$. We show that in fact $\coker f$ is annihilated by an element of $A$ (as opposed to $A_L$); see \ref{Cor1.11b}. Therefore every isogeny possesses a dual (\ref{Cor1.11b}) and the group of quasi-isogenies equals the group of units in the endomorphism $Q$-algebra (\ref{QISOG-GROUP}). We give various other descriptions for \mbox{(quasi-)}isogenies (\ref{PROP.1.42A}, \ref{PropAltDescrQHom}). Also we prove that the existence of a separable isogeny defines an equivalence relation on pure $A$-motives over a finite field (\ref{ThmW5.2}), but not over an infinite field (\ref{Ex8.10}). 

We develop the theory of Tate modules and local shtukas. The later are the analogs of Dieudonn\'e modules for (the $p$-divisible groups of) abelian varieties, except that $p$-divisible groups are only useful for abelian varieties in characteristic $p$, whereas the local shtukas at any place of $Q$ are important for the investigation of abelian $\tau$-sheaves and pure $A$-motives. We prove the standard facts on the relation between Tate modules and isogenies (\ref{Prop2.7b}, \ref{Prop1.5b}, \ref{Prop1.5c}). Also we use local shtukas to give a proof of the fact that $\Hom(\ulM,\ulM')$ is a projective $A$-module of rank $\le rr'$. 
In a continuation of this article we study in \cite{BH_B} the behavior of pure $A$-motives over finite fields and obtain answers which are similar to Tate's famous results \cite{Tat} for abelian varieties.
There is a two in one version \cite{BH} of the present article and \cite{BH_B}
on the arXiv.


\subsection*{Notation}

In this article we denote by

\vspace{2mm}
\noindent
\begin{tabular}{@{}p{0.25\textwidth}@{}p{0.75\textwidth}@{}}
$\Fq$& the finite field with $q$ elements and characteristic $p$, \\
$C$& a smooth projective geometrically irreducible curve over $\Fq$, \\
$\infty\in C(\Fq)$& a fixed $\Fq$-rational point on $C$, \\
$A = \Gamma(C\setminus\{\infty\},\OC)$& the ring of regular functions on $C$ outside $\infty$, \\
$Q = \Fq(C)= \Quot(A)$& the function field of $C$, \\
$Q_v$& the completion of $Q$ at the place $v\in C$, \\
$A_v$& the ring of integers in $Q_v$. For $v\ne\infty$ it is the completion of $A$ at $v$.\\
$\BF_v$ & the residue field of $A_v$. In particular $\BF_\infty\cong\Fq$.
\end{tabular}

\vspace{2mm}
\noindent
For a field $L$ containing $\Fq$ we write

\vspace{2mm}
\noindent
\begin{tabular}{@{}p{0.25\textwidth}@{}p{0.75\textwidth}@{}}
$C_L=C\times_{\Spec\Fq}\Spec L$,\\[1mm]
$A_L=A\otimes_\Fq L$,\\[1mm]
$Q_L=Q\otimes_\Fq L$,\\[1mm]
$A_{v,L}=A_v\wh\otimes_\Fq L$ & for the completion of $\O_{C_L}$ at the closed subscheme $v\times\Spec L$,\\[1mm]
$Q_{v,L}=A_{v,L}[\frac{1}{v}]$.& Note that this is not a field if $\BF_v\cap
L\supsetneq\Fq$. Nevertheless, it is always a finite product of fields.\\[1mm]
\end{tabular}
\begin{tabular}{@{}p{0.25\textwidth}@{}p{0.75\textwidth}@{}}
${\rm Frob}_q:L\to L$ & for the $q$-Frobenius endomorphism mapping $x$ to $x^q$,\\[1mm]
$\sigma = \id_C\times\Spec({\rm Frob}_q)$& for the endomorphism of $C_L$ which acts as the identity on the points and on $\O_C$ and as the $q$-Frobenius on $L$,\\
$\s$ & for the endomorphisms induced by $\sigma$ on all the above rings. For instance $\s(a\otimes b)=a\otimes b^q$ for $a\in A$ and $b\in L$.\\
$\s M=M\otimes_{A_L,\s}A_L$ & for an $A_L$-module $M$ and similarly for the other rings.
\end{tabular}

\forget{
\vspace{2mm}
\noindent
All schemes, as well as their products and morphisms, are supposed to be over $\Spec\Fq$. Let $S$ be a scheme. We denote by

\vspace{2mm}
\noindent
\begin{tabular}{@{}p{0.25\textwidth}@{}p{0.75\textwidth}@{}}
$\sigma_S: S\rightarrow S$& its $q$-Frobenius endomorphism which acts identically on points of $S$ and as the $q$-power map on the structure sheaf $\OS$, \\
$C_S = C\times S$& the base extension of $C$ from $\Spec\Fq$ to $S$, \\
$\sigma = \id_C\times\sigma_S$& the endomorphism on $C_S$ which acts identically on $C$ and as the $q$-Frobenius on $S$. 
\end{tabular}
}

\vspace{2mm}
\noindent
For a divisor $D$ on $C$ we denote by $\O_{C_L}(D)$ the invertible sheaf on $C_L$ whose sections $\varphi$ have divisor $(\varphi)\ge-D$. 
For a coherent sheaf $\F$ on $C_L$ we set $\F(D) := \F\otimes_{\O_{C_L}}\O_{C_L}(D)$. This notation applies in particular to the divisor $D = n\cdot\infty$ for $n\in\Z$.


\vspace{1cm}


\section{Pure $A$-motives}

Let $L$ be a field extension of $\Fq$ and fix an $\Fq$-homomorphism $c^\ast:A\to L$. Let $J\subset A_L$ be the ideal generated by $a\otimes 1-1\otimes c^\ast(a)$ for all $a\in A$, and let $r$ and $d$ be non-negative integers. Pure $A$-motives were introduced by G.\ Anderson~\cite{Anderson} and called \emph{pure $t$-motives} in the case where $A=\BF_p[t]$. In the slightly more general case we define:

\begin{Definition}[pure $A$-motives] \label{Def1'.1}
A \emph{pure $A$-motive} (or also \emph{pure Anderson motive}) \emph{$\ulM=(M,\tau)$ of rank $r$, dimension $d$, and characteristic $c^\ast$ over $L$} consists of a locally free $A_L$-module $M$ of rank $r$ and an injective $A_L$-module homomorphism $\tau:\s M\to M$ such that
\begin{suchthat}
\item 
the cokernel of $\t$ is an $L$-vector space of dimension $d$ and annihilated by $J^d$, and 
\item 
$M$ extends to a locally free sheaf $\CM$ of rank $r$ on $C_L$ such that for some positive integers $k,l$ the map $\tau^l:=\tau\circ\s(\tau)\circ\ldots\circ(\s)^{l-1}(\tau):(\s)^l M\to M$ induces an isomorphism $(\s)^l\CM_\infty\to\CM(k\cdot\infty)_\infty$ of the stalks at $\infty$. 
\end{suchthat}
We call $\chr:=\ker c^\ast\in \Spec A$ the \emph{characteristic point} of $\ulM$. We say that $\ulM$ has \emph{finite characteristic} (respectively \emph{generic characteristic}) if $\chr$ is a closed (respectively the generic) point. For $r>0$ we call $\weight(M,\tau):=\frac{k}{l}$ the \emph{weight} of $(M,\tau)$.
\end{Definition}

\noindent {\it Remark.} \nopagebreak 
Phrased in the language of modules our definition of purity is equivalent to the following due to Anderson~\cite[1.9]{Anderson}. Let $z$ be a uniformizing parameter of $A_{\infty,L}$. The $A$-motive is pure if and only if there exists an $A_{\infty,L}$-lattice $\hat\CM_\infty$ inside $M\otimes_{A_L}Q_{\infty,L}$ and positive integers $k,l$ such that $z^k\tau^l$ induces an isomorphism $(\s)^l\hat\CM_\infty\to\hat\CM_\infty$. This follows from the fact that $\hat\CM_\infty$ determines a unique extension $\CM$ of $M$ as above.

\begin{Proposition}\label{Prop1'.1b}
If $\ulM$ is a pure $A$-motive of rank $r>0$ then $\weight\ulM=\frac{d}{r}$\,. In particular $\dim\ulM>0$.
\end{Proposition}

\begin{proof}
Using \ref{DEGREE-LEMMA} below we compute
\[
k\,r \; = \; \deg \CM(k\cdot\infty)-\deg\CM
\; = \; \deg\CM(k\cdot\infty)-\deg(\s)^l\CM
\; = \;\dim_L\coker(\tau^l) \;=\; l\,d\,.\qedhere
\]
\end{proof}

\begin{Lemma}\label{DEGREE-LEMMA}
Let\/ $\G$ be a coherent sheaf on $C_L$. Then $\deg\s\G = \deg\G$.
\end{Lemma}

\begin{proof}
Let $pr:C_L\to \Spec L$ be the projection onto the second factor. If $\G=\OCL^{\!\oplus n}$ for some $n\in\N$, then $\s\G = \OCL^{\!\oplus n}$. If $\G=\OCL(D)$ for some divisor $D$ on $C$, then $\s\G=\OCL(\sigma^{-1}D)=\OCL(D)$. If $\G$ is a torsion sheaf we get
\[
\deg\s\G \,=\, \dim_L pr_\ast\s\G \,=\, \dim_L \s pr_\ast\G \,=\, \dim_L pr_\ast\G \,=\, \deg\G\,.
\]
Let now $\G$ be a locally free sheaf of rank $n$. Choose an embedding $f:\,\G\rightarrow\OCL(D)^{\oplus n}$ for some divisor $D$ on $C$ with $\coker f$ being a torsion sheaf. Since $\sigma=\id_C\times\sigma_L$ is flat being the base change of the flat morphism $\sigma_L:\Spec L\to\Spec L$, we have 
\[
\begin{CD} 
{0} @>>> {\G}   @>{f   }>> {\OCL(D)^{\oplus n}}   @>>> {\coker f}   @>>> {0} \\[1ex]
{0} @>>> {\s\G} @>{\s f}>> {\s\OCL(D)^{\oplus n}} @>>> {\s\coker f} @>>> {0} \\[1ex]
\end{CD}
\]
and therefore $\deg\s\G = \deg\G$ due to the additivity of the degree in exact sequences. Finally, if $\G$ is an arbitrary coherent sheaf, then
\[
\exact{0}{}{\G'}{}{\G}{}{\G''}{}{0}
\]
for some torsion sheaf $\G'$ and some locally free coherent sheaf $\G''$ because this sequence exists locally due to the fact that all local rings are principal ideal domains. Thus $\deg\s\G = \deg\G$, as desired.
\end{proof}

\begin{Proposition} \label{PropX.3}
If $(M,\tau)$ is a pure $A$-motive over $L$ then one can find an extension $\CM$ as above with $k$ and $l$ relatively prime.
\end{Proposition}

\begin{proof}
We let $z$ be a uniformizing parameter at $\infty$ and write $\frac{d}{r}=\frac{k}{l}$ with $k,l$ relatively prime positive
integers. Since $(M,\t)$ is pure it extends to a locally free sheaf $\CM'$ on $C_L$ on which $z^{k'}\t^{l'}$ is an isomorphism locally at $\infty$ for some positive integers $k',l'$ with $\frac{k'}{l'}=\frac{d}{r}=\frac{k}{l}$. We modify $\CM'$ to a locally free sheaf $\CM$ on $C_L$ by changing its stalk $\CM'_\infty$ at $\infty$ to
\[
\CM_\infty\;:=\;\sum_{j=0}^{\frac{l'}{l}-1}z^{kj}\t^{lj}\bigl((\s)^{lj}\CM'_\infty\bigr)\,.
\]
Then $M=\Gamma(C_L\setminus\{\infty\},\CM)$ and $z^{k}\tau^{l}:(\s)^{l}\CM_\infty\isoto\CM_\infty$ is an isomorphism at $\infty$ as desired. 
\forget{
We write $\frac{d}{r}=\frac{k'}{l'}$ with $k',l'$ relatively prime positive integers and set
\[
\ulN_\infty := (N_\infty,\tau):=(M,\tau)\otimes_{A_L}Q_{\infty,L}\,.
\]
In the terminology of \cite[Definition~7.3]{Hl} the pair $(N_\infty,\tau)$ is a \emph{Dieudonn\'e $\BF_q\dpl z\dpr$-module} over $L$ (also called a \emph{local isoshtuka} in Section~\ref{SectLS} below). Due to the purity condition there exists an $A_{\infty,L}$-lattice $\hat\CM_\infty$ inside $N_\infty$ and positive integers $k,l$ with $\frac{k}{l}=\frac{d}{r}$ such that $z^{k}\tau^{l}$ is an isomorphism $(\s)^l\hat\CM_\infty\to\hat\CM_\infty$. Therefore $\ulN_\infty$ is isoclinic of slope $\frac{k}{l}=\frac{k'}{l'}$. Consider the $A_{\infty,L}$-lattice
\[
\hat\CM'_\infty\;:=\;\sum_{j=0}^{\frac{l}{l'}-1}z^{k'j}\t^{l'j}\bigl((\s)^{l'j}\hat\CM_\infty\bigr)
\]
inside $\ulN_\infty$ on which $z^{k'}\tau^{l'}:(\s)^{l'}\hat\CM'_\infty\isoto\hat\CM'_\infty$ is an isomorphism. We glue $M$ with $\hat\CM'_\infty$ to the locally free sheaf $\CM'$ on $C_L$ on which now $(\tau')^{l'}:(\s)^{l'}\CM'\to \CM(k'\cdot\infty)$ is an isomorphism near $\infty$. 
}
\end{proof}

\begin{Definition} (compare \cite[4.5]{PT})\label{Def1'.2}
\begin{suchthat}
\item 
A \emph{morphism} $f:(M,\tau)\to (M',\tau')$ between $A$-motives of the same characteristic $c^\ast$ is a homomorphism $f:M\to M'$ of $A_L$-modules which satisfies $f\circ\tau=\tau'\circ\s(f)$.
\item 
If $f:\ulM\to\ulM'$ is surjective, $\ulM'$ is called a \emph{quotient} (or \emph{factor}) \emph{motive} of $\ulM$.
\item 
A morphism $f:\ulM\to \ulM'$ is called an \emph{isogeny} if $f$ is injective with torsion cokernel.
\item 
An isogeny is called \emph{separable} (respectively \emph{purely inseparable}) if the induced homomorphism $\t:\s\coker f\to\coker f$ is an isomorphism (respectively is \emph{nilpotent}, that is, if $\t^n=0$ for some $n$).
\end{suchthat}
\end{Definition}

\bigskip

\noindent {\it Remark.} \es
1. The set $\Hom(\ulM,\ulM')$ of morphisms is an $A$-module and $\End(\ulM)$ is an $A$-algebra. They are projective $A$-modules of rank $\le rr'$. This will be proved in Theorem~\ref{ThmT.3}.

\smallskip
\noindent
2. One has $\Hom(\ulM,\ulM')=\{0\}$ if $\ulM$ and $\ulM'$ are pure $A$-motives of different weights, justifying the terminology \emph{pure}. This can be derived from the Dieudonn\'e-Manin type classification \cite[Appendix B]{Laumon} of the local $\sigma$-isoshtuka $\ulM_\infty(\ulM):=\ulM\otimes_{A_L}Q_{\infty,L}$ of $\ulM$ at $\infty$; see Section~\ref{SectLS}. However, we will give a more elementary proof in Corollary~\ref{Cor2.9b} below.

\smallskip
\noindent
3. We will prove in Corollary~\ref{Cor1.11b} below that the cokernel of an isogeny is in fact annihilated by an element $a\in A$ (as opposed to $a\in A_L$). This was independently observed by N.~Stalder~\cite{Stalder} and also holds for non-pure $A$-motives.

\begin{Proposition}\label{Prop1.5b}
Let $(M,\t)$ be a pure $A$-motive and let $K$ be a finite torsion $A_L$-module equipped with an $A_L$-homomorphism $\t_K:\s K\to K$ such that both $\ker\t_K$ and $\coker\t_K$ are annihilated by a power of $J=(a\otimes1-1\otimes c^\ast(a):a\in A)\subset A_L$. Let further $\rho:M\shortonto K$ be a surjective morphism of $A_L$-modules with $\t_K\circ\s\rho=\rho\circ\t$. Then $(M',\t'):=(\ker\rho,\t|_{\s M'})$ is again a pure $A$-motive of the same rank and dimension and the inclusion $f:(M',\t')\to(M,\t)$ is an isogeny with $\coker f=(K,\t_K)$.
\end{Proposition}

\begin{proof}
Consider the diagram in which the bottom row is obtained from the snake lemma
\[
\xymatrix {& 0 \ar[r] & \s M' \ar[r]^{\s f} \ar@{^{ (}->}[d]_{\t'} & \s M \ar[r]^{\s\rho} \ar@{^{ (}->}[d]_{\t} & \s K \ar[r] \ar[d]_{\t_K} & 0 \\
& 0 \ar[r] & M' \ar[r]^f \ar@{->>}[d] & M \ar[r]^\rho \ar@{->>}[d] & K \ar[r] \ar@{->>}[d] & 0 \\
0 \ar[r] & \ker\t_K \ar[r] & \coker\t' \ar[r] & \coker\t \ar[r] & \coker\t_K \ar[r] & 0\;.}
\]
If follows that $\dim_L\coker\t'=\dim_L\coker\t=d$ and that also $\coker\t'$ is annihilated by $J^d$. The purity follows from the fact that one can extend $f$ to an isomorphism $\CM'_\infty\to\CM_\infty$ of the stalks at infinity.
\end{proof}

\noindent
{\it Remark.}
Note that without the requirement that a power of $J$ annihilates $\ker\tau_K$ and $\coker\tau_K$ the assertion  of the proposition is false as one can see from $A=\BF_q[t],\,M=A_L,\,\tau=t\otimes1-1\otimes c^\ast(t),\,K=\coker\tau,\,\tau_K=0$, when $c^\ast(t)^q\ne c^\ast(t)$.

\begin{Lemma}\label{Lemma1.5d}
Let $K$ be a finite torsion $A_L$-module and let $\t:\s K\to K$ be a morphism of $A_L$-modules. Then $\ulK=(K,\t)$ is an extension
\[
0\longto \ulK^\et\xrightarrow{\es f\;} \ulK\xrightarrow{\es g\;}\ulK^\nil\longto 0
\]
of $\ulK^\nil=(K^\nil,\t^\nil:\s K^\nil\to K^\nil)$ by $\ulK^\et=(K^\et,\t^\et:\s K^\et\to K^\et)$ where $\t^\nil$ is nilpotent and $\t^\et$ is an isomorphism, satisfying $f\circ\t^\et=\t\circ\s f$ and $g\circ\t=\t^\nil\circ\s g$. Moreover, if the base field $L$ is perfect the extension splits canonically.
\end{Lemma}

\begin{proof}
This was proved by Laumon~\cite[B.3.10]{Laumon}. He takes $ K^\et:=\bigcap_{n\ge1}\im\t^n$. If $L$ is perfect, $K^\et$ has the natural complement $\bigcup_{n\ge1}(\s)^{-n}(\ker\t^n)$ which is isomorphic to $\ulK^\nil=\ulK/\ulK^\et$.
\end{proof}

\begin{Proposition}\label{Prop1.5c}
Every isogeny $f:\ulM\to \ulM'$ can be factored $\ulM\xrightarrow{\es f_{\rm sep}\;}\ulM''\xrightarrow{\es f_{\rm insep}\;}\ulM'$ into a separable isogeny $f_{\rm sep}$ followed by a purely inseparable isogeny $f_{\rm insep}$. If the base field is perfect there exists also a (different) factorization $f=f'_{\rm sep}\circ f'_{\rm insep}$ as a purely inseparable isogeny followed by a separable one.
\end{Proposition}

\begin{proof}
Let $K:=\coker f$ and let $\t_K:\s K\to K$ be the induced morphism. By Lemma~\ref{Lemma1.5d} there is a surjective morphism $\rho:\ulM'\to\ulK\to\ulK^\nil$ and we define $\ulM''$ as the kernel of $\rho$. It is a pure $A$-motive by Proposition~\ref{Prop1.5b}. Clearly $f$ factors through $\ulM''$ and the isogeny $\ulM\to \ulM''$ has $\ulK^\et$ as cokernel, thus is separable. If $L$ is perfect we use the surjective morphism $\rho:\ulM'\to\ulK\to\ulK^\et$ instead.
\end{proof}


\bigskip
\section{Definition of abelian $\tau$-sheaves}

Since the purity condition 2 of Definition~\ref{Def1'.1} does not behave well in families one has to rigidify $\ulM$ at $\infty$ in order to get moduli spaces for pure $A$-motives. This was done in \cite{Hl}, where the rigidified objects are called \emph{abelian sheaves}. Over a field their definition is as follows.
Let $L\supset\Fq$ be a field and fix a morphism $c: \Spec L\rightarrow C$. Let $\I$ be the ideal sheaf on $C_L$ of the graph $\Graph(c)$ of $c$. Let $r$ and $d$ be non-negative integers.

\begin{Definition}[Abelian $\tau$-sheaf]\label{Def1.1}
An \emph{abelian $\tau$-sheaf $\FF = (\F_i,\P_i,\t_i)$ of rank $r$, dimension $d$ and characteristic $c$ over $L$} is a collection of locally free sheaves $\F_i$ on $C_L$ of rank $r$ together with injective morphisms $\P_i,\t_i$ of\/ $\O_{C_L}$-modules $(i\in\Z)$ of the form
\[
\begin{CD}
\cdots @>>> 
\F_{i-1} @>{\P_{i-1}}>> 
\F_{i}   @>{\P_{i}}>> 
\F_{i+1} @>{\P_{i+1}}>> 
\cdots 
\\ & & 
@AA{\t_{i-2}}A 
@AA{\t_{i-1}}A 
@AA{\t_{i}}A 
\\
\cdots @>>> 
\s\F_{i-2} @>{\s\P_{i-2}}>> 
\s\F_{i-1} @>{\s\P_{i-1}}>> 
\s\F_{i}   @>{\s\P_{i}}>> 
\cdots 
\\
\end{CD}
\]
subject to the following conditions:
\begin{suchthat}
\item the above diagram is commutative,
\item there exist integers\/ $k,l>0$ with\/ $ld=kr$ such that the morphism $\P_{i+l-1}\circ\cdots\circ\P_i$ identifies $\F_i$ with the subsheaf\/ $\F_{i+l}(-k\cdot\infty)$ of\/ $\F_{i+l}$ for all $i\in\Z$,
\item $\coker\P_i$ considered as an $L$-vector space has dimension $d$ for all $i\in\Z$,
\item $\coker\t_i$ is annihilated by $\I^d$ and as an $L$-vector space has dimension $d$ and for all $i\in\Z$.
\end{suchthat}
We call $\chr:=c(\Spec L)\in C$ the \emph{characteristic point} (or \emph{place}) and say that $\FF$ has \emph{finite} (respectively \emph{generic}) \emph{characteristic} if $\chr$ is a closed (respectively the generic) point. 
\end{Definition}

\begin{Remark} \label{Rem2.2}
1. By the second condition $\coker\P_i$ is only supported at $\infty$. Moreover, the periodicity condition implies $\F_{i+nl} = \F_i(nk\cdot\infty)$ and thus $\t_{i+nl}=\t_i\otimesidOCL{nk\cdot\infty}$ for all $n\in\Z$.

\smallskip
\noindent
2. The condition ``annihilated by $\I$'' in 4 can equivalently be reduced to ``supported on the graph of $c$ '', since the local ring of $C_L$ at the graph of $c$ is a principal ideal domain and the $d$-th power of a generator of $\I$ annihilates the $d$-dimensional $L$-vector space $\coker\t_i$.

\smallskip
\noindent
3. Trivially, $r=0$ implies $d=0$ since in this case we have all $\F_i=0$. Due to the second condition, the converse is also true because $d=0$ implies $r = \frac{l}{k}\,d = 0$ since the existence of such $k,l\not= 0$ is required. Without this the converse would in general not be true, because for example $\FF = (\OCS,\id_\OCS,\id_\OCS)$ has $\coker\id_\OCS = 0$ and therefore $d=0$, but $r=1$. This justifies the demand of the existence of such $k,l\not= 0$ since we do not want to consider the \quotes{degenerate} case $r>0$, $d=0$.

The case $r=0$, $d=0$ however is desired because it allows the \emph{zero sheaf} $\ZZ := (0,0,0)$ to be an abelian $\tau$-sheaf of rank $0$ and dimension $0$. Trivially, the zero sheaf satisfies the second condition for \emph{all} pairs $k,l > 0$.

\smallskip
\noindent
4. For $\FF\not= \ZZ$ one can ask whether the second condition is satisfied by the pair $k,l > 0$ with $ld = kr$ and $k,l$ relatively prime. This was required in the definition of abelian $\tau$-sheaves in \cite{Hl}. We will call those $\FF$ \emph{abelian $\tau$-sheaves with $k,l$ relatively prime}.
As a convention, an abelian $\tau$-sheaf $\FF$ without further specifications comes with all its parameters $\F_i,\P_i,\t_i$ $(i\in\Z)$ and $r,d,k,l$ with $k,l$ always chosen to be minimal. Similarly $\FF'$ carries a prime on its parameters, $\TFF$ a tilde on them, and so on. Note that the characteristic $c$ is fixed.

\smallskip
\noindent
5. Abelian $\tau$-sheaves of dimension $d=1$ are called \emph{elliptic sheaves} and were studied by Drinfeld~\cite{Drinfeld3}, Blum-Stuhler~\cite{BS} and others. The category of elliptic sheaves with $(k,l)=1$ over $L$ of rank $r$ with $\deg\F_0=1-r$ and whose characteristic does not meet $\infty$, that is, $\im(c)\subset C\setminus\{\infty\}$, is anti-equivalent to the category of Drinfeld-$A$-modules of rank $r$ over $L$, see \cite[Theorem~3.2.1]{BS} and Example~\ref{Ex1.8} below.
\end{Remark}

\smallskip

\begin{Definition}\label{DEFINITION-WEIGHT}
Let\/ $\FF$ be an abelian $\tau$-sheaf of rank $r$, dimension $d$ and characteristic $c$ over $S$. We set
\[
\weight(\FF) \;:=\; 
\left\{\;
\begin{array}{cl} 
\frac{d}{r} & \mbox{if\/ }\FF\not=\ZZ \\[1ex] 
0           & \mbox{if\/ }\FF=\ZZ 
\end{array} 
\;\right\}
\;\in\Q\,.
\]
We call\/ $\weight(\FF)$ the \emph{weight of $\FF$}.
\end{Definition}

\begin{Example} \label{Ex1.8}
Let  $C=\PP^1_\Fq$, $A=\FqItI$. Then $\CL=\PP^1_L$. Let $c: \Spec L\rightarrow \Spec\FqItI=\AA^1_\Fq$ such that $c^\ast: \FqItI\rightarrow L$ maps $t$ to $c^\ast(t)=:\vartheta$. Let $\O$ denote the structure sheaf of $\PP^1_L$ and let $a\in L$. Now consider the following diagram
\[
\begin{CD}
\cdots @>>> 
\O\oplus\O @>{\matr{1}{0}{0}{1}}>{\P_0}> 
\O\oplus\O(1\!\cdot\!\infty) @>{\matr{1}{0}{0}{1}}>{\P_1}> 
\O(1\!\cdot\!\infty)\oplus\O(1\!\cdot\!\infty) @>{\matr{1}{0}{0}{1}}>{\P_2}>
\cdots 
\\ & & 
@A{\textstyle\t_{\scriptscriptstyle-1}}\;AA
@A{\textstyle\t_{\scriptscriptstyle 0}}\;A\;{\matr{a}{t-\vartheta}{1}{0}}A
@A{\textstyle\t_{\scriptscriptstyle 1}}\;A\;{\matr{a}{t-\vartheta}{1}{0}}A 
\\
& &
\cdots @>{\matr{1}{0}{0}{1}}>{\s\P_{-1}}> 
\s(\O\oplus\O) @>{\matr{1}{0}{0}{1}}>{\s\P_0}> 
\s(\O\oplus\O(1\!\cdot\!\infty)) @>{\matr{1}{0}{0}{1}}>{\s\P_1}> 
\cdots 
\\
\end{CD}
\]
where the vectors in $\O\oplus\O$ are considered as column vectors.
This gives an example of an abelian $\tau$-sheaf $\FF$ of rank $2$, dimension $1$ with $(k,l)=1$, and characteristic $c$ over $\Spec L$ with $\weight(\FF)=\frac{1}{2}$.

Since the dimension is $1$, this abelian $\tau$-sheaf is an elliptic sheaf and comes from a Drinfeld module which can be recovered as follows; see \cite{BS}. Let $M := \Gamma(\AA^1_L,\O\oplus\O) = L[t]\oplus L[t]$ and let $\t := \P_0^{-1}\circ\t_0$. Since $\vect{1}{0}=\t\vect{0}{1}$, we have $M=L\{\t\}\cdot\vect{0}{1}$, and we calculate
\[
\t^2\tvect{0}{1} \,=\, \tvect{a}{t-\vartheta} \,=\, a\cdot\t\tvect{0}{1} + (t-\vartheta)\cdot\tvect{0}{1}
\quad \text{and} \quad
t\cdot\tvect{0}{1} = (\vartheta - a\t + \t^2)\tvect{0}{1}
\]
Let $\varphi: \FqItI\rightarrow L\{\t\}$ be the ring morphism mapping $t\mapsto \vartheta - a\t + \t^2$. Then we have back the Drinfeld Module $\varphi$ of rank 2 over $L$ which induces the abelian $\tau$-sheaf $\FF$.
\end{Example}

\begin{Example} \label{Ex1.8b}
We give another example which does not come from Drinfeld modules. Let $C=\PP^1_\Fq$ and $A=\Fq[t]$. Let 
\[
S=\Spec \Fq[\zeta,\alpha,\beta,\gamma,\delta]\,/\,\bigl((\alpha+\delta+2\zeta)^2\,,\, \alpha\delta-\beta\gamma-\zeta^2\bigr)
\]
and $c:S\to C\setminus V(t)\subset C$ be given by $c^\ast(\frac{1}{t})=\zeta$. Let $C_S:=C\times_{\Spec\Fq}S$, then
\[
\begin{CD}
\cdots @>>> 
\O_{C_S}^{\oplus2} @>{\quad\matr{1}{0}{0}{1}\quad}>{\P_0}> 
\O_{C_S}(1\!\cdot\!\infty)^{\oplus2} @>{\qquad\matr{1}{0}{0}{1}\qquad}>{\P_1}> 
\O_{C_S}(2\!\cdot\!\infty)^{\oplus2} @>{\quad\matr{1}{0}{0}{1}\quad}>{\P_2}>
\cdots 
\\ & & 
@A{\textstyle\t_{\scriptscriptstyle-1}}\;AA
@A{\textstyle\t_{\scriptscriptstyle 0}}\;A\;{\matr{1+\alpha t}{\gamma t}{\beta t}{1+\delta t}}A
@A{\textstyle\t_{\scriptscriptstyle 1}}\;A\;{\matr{1+\alpha t}{\gamma t}{\beta t}{1+\delta t}}A 
\\
& &
\cdots @>{\quad\matr{1}{0}{0}{1}\quad}>{\s\P_{-1}}> 
\s\O_{C_S}^{\oplus2} @>{\qquad\matr{1}{0}{0}{1}\qquad}>{\s\P_0}> 
\s\O_{C_S}(1\!\cdot\!\infty)^{\oplus2} @>{\quad\matr{1}{0}{0}{1}\quad}>{\s\P_1}> 
\cdots 
\\
\end{CD}
\]
is an abelian $\tau$-sheaf over $S$ of rank and dimension $2$ with $k=l=1$ since $(1-\zeta t)^2\cdot\coker\tau=(0)$. In fact $S$ is the (representable part of the) moduli space of abelian $\tau$-sheaves of rank and dimension $2$ with $(k,l)=1$ together with a level structure $\eta$ at $V(t)$ for which $(\F_0,\eta)$ is stable of degree zero. See~\cite[\S 4]{Hl} for the precise meaning of these terms, but note that in loc.\ cit.\ the exponent $2$ in $(\alpha+\delta+2\zeta)^2$ erroneously was missing, as was pointed out to us by M.~Molz. This illustrates the fact that abelian $\tau$-sheaves possess nice moduli spaces.
\end{Example}

\begin{Proposition}\label{PROP.2}
Let $\FF$ be an abelian $\tau$-sheaf and let $D$ be a divisor on $C$. Then the collection $\FF(D) := (\F_i(D),\P_i\otimesidOCL{D},\t_i\otimesidOCL{D})$ is an abelian $\tau$-sheaf of the same rank and dimension as $\FF$.
\end{Proposition}

\begin{proof}
Since the functor $\otimes_{\O_{C_L}}\O_{C_L}(D)$ is exact the proof is straightforward ones one notes that $\s(\F_i(D))=(\s\F_i)(D)$ because the divisor $D$ is $\sigma$-invariant.
\end{proof}

Next we come to the definition of morphisms in the category of abelian $\tau$-sheaves.

\begin{Definition}
A \emph{morphism} $f$ between two abelian $\tau$-sheaves $\FF = (\F_i,\P_i,\t_i)$ and $\FF' = (\F'_i,\P'_i,\t'_i)$ of the same characteristic $c:\Spec L\to C$ is a collection of morphisms $f_i: \F_i \rightarrow \F'_i$ $(i\in\Z)$ which commute with the $\P$'s and the $\t$'s, that is, $f_{i+1}\circ\P_i=\P'_i\circ f_i$ and $f_{i+1}\circ\t_i=\t'_i\circ\s f_i$. We denote the set of morphisms between $\FF$ and $\FF'$ by $\Hom(\FF,\FF')$. It is an $\Fq$-vector space.
\end{Definition}

\begin{Definition}
Let $\FF$ and $\FF'$ be abelian $\tau$-sheaves and let $f \in \Hom(\FF,\FF')$ be a morphism. 
Then $f$ is called \emph{injective} (respectively \emph{surjective}, respectively an \emph{isomorphism}), if\/ $f_i$ is injective (respectively surjective, respectively bijective) for all\/ $i\in\Z$.
We call $\FF$ an \emph{abelian quotient} (or \emph{factor}) \emph{$\tau$-sheaf} of\/ $\FF'$, if there is a surjective morphism from $\FF'$ onto $\FF$.
\end{Definition}

Abelian $\tau$-sheaves are pure in the following sense.

\begin{Proposition}\label{PROP.1}
Let $\FF$ and $\FF'$ be abelian $\tau$-sheaves. If\/ $\Hom(\FF,\FF')\not=\{0\}$, then\/ $\weight(\FF)=\weight(\FF')$. 
\end{Proposition}

\begin{proof}
Let $0\not=f\in\Hom(\F,\F')$ and let $i\in\Z$. Consider the sheaf $\HOM(\F^{\,}_i,\F'_i) = \F'_i\otimes_\OCL\dual{\F}_i$ and the set of all its locally free subsheaves $\M\subset\F'_i\otimes_\OCL\dual{\F}_i$. Then the set of their degrees $\deg\M$ is bounded above, say with upper bound $B$ by \cite[Lemma 1.I.3]{Se}.

Suppose $d'r<dr'$. Choose $n\in\Z$ with $ll'|nrr'$ such that $B+n(d'r-dr')<0$. Let $\P$ and $\P'$ be the identifying morphisms $\P_{i+nrr'-1}\circ\cdots\circ\P_i: \F_i \cong \F_{i+nrr'}(-ndr'\cdot\infty)$ and $\P'_{i+nrr'-1}\circ\cdots\circ\P'_i: \F'_i \cong \F'_{i+nrr'}(-nd'r\cdot\infty)$, respectively. Consider the following diagram
\[
\xymatrix{
\F_i\; \ar[r]^{\P\;\;}\ar[d]^{f_i}& \;\F_{i+nrr'} \makebox[0pt][l]{ $= \;\F_i(ndr'\cdot\infty)\;$} \ar[d]^{f_{i+nrr'}} &\qquad \\
\F'_i\; \ar[r]^{\P'\;\;}& \;\F'_{i+nrr'} \makebox[0pt][l]{ $= \;\F'_i(nd'r\cdot\infty)\;$\,.} &\qquad \\
}
\]
With $m:=n(d'r-dr')<0$, we conclude that $\HOM(\F_{i+nrr'},\F'_{i+nrr'}) \;=\; 
(\F'_i\otimes_\OCL\dual{\F}_i)(m\cdot\infty)$.
Now, considering the injective map $\varphi: \OCL\rightarrow(\F'_i\otimes_\OCL\dual{\F}_i)(m\cdot\infty)$, $1\mapsto f_{i+nrr'}$ we get a non-zero locally free subsheaf $\im\varphi=\M(m\cdot\infty)$ which is isomorphic to $\OCL$ with $\M\subset\F'_i\otimes_\OCL\dual{\F}_i$. Therefore 
\[
0\;=\;\deg\M(m\cdot\infty) \;=\; \deg\M+m\cdot\rank\M \;\le\; B + m \;<\; 0\,.
\]
This is a contradiction and shows $d'r\ge dr'$. The converse $d'r\le dr'$ follows analogously.
\end{proof}

\noindent
{\it Remark.} This result can also be proved by considering the local isoshtukas at $\infty$ of $\FF,\FF'$ (see Section~\ref{SectLS}) and using the Dieudonn\'e-Manin theory \cite[Appendix B]{Laumon} for local isoshtukas.


\bigskip

\section{Relation between pure $A$-motives and abelian $\tau$-sheaves} \label{SectRelation}

If $\FF=(\F_i,\P_i,\t_i)$ is an abelian $\tau$-sheaf of rank $r$, dimension $d$, and characteristic $c:\Spec L\to C$ with characteristic place $\chr=\im c\ne\infty$ then
\[
\ulM(\FF)\;:=\;(M,\tau)\;:=\;\Bigl(\Gamma(C_L\setminus\{\infty\},\F_0)\,,\,\P_0^{-1}\circ\t_0\Bigr)
\]
is a pure $A$-motive of the same rank and dimension and of characteristic $c^\ast:A\to L$. We can take $\CM:=\F_0$ as the extension of $M$ to all of $C_L$. Conversely we have the following result.

\begin{Theorem}\label{PropX.1}
\begin{enumerate}
\item 
Let $(M,\tau)$ be a pure $A$-motive of rank $r$, dimension $d$, and characteristic $c^\ast:A\to L$ over $L$. Then $(M,\tau)=\ulM(\FF)$ for an abelian $\tau$-sheaf $\FF$ over $L$ of same rank and dimension with characteristic $c:=\Spec c^\ast:\Spec L\to \Spec A\subset C$. One can even find the abelian $\tau$-sheaf $\FF$ with $k,l$ relatively prime.
\item 
Let $\FF$ and $\FF'$ be abelian $\tau$-sheaves of the same weight and let $f_0:\ulM(\FF)\to \ulM(\FF')$ be a morphism. Then there exists an integer $m$ such that $f_0$ comes from a morphism $f:\FF\to\FF'(m\cdot\infty)$ as $f_0=\ulM(f)$.
\end{enumerate}
\end{Theorem}

\begin{proof}
1. Let $\CM$ be a locally free sheaf on $C_L$ with $M=\Gamma(C_L\setminus\{\infty\},\CM)$ as in Definition \ref{Def1'.1}. Let $k,l$ be positive integers with isomorphism
\[
\tau^l:(\s)^l\CM_\infty\isoto\CM(k\cdot\infty)_\infty\,.
\]
By Proposition~\ref{PropX.3} we may assume $(k,l)=1$.

For $i=0,\ldots,l$ let $\F_i$ be the locally free sheaf of rank $r$ on $C_L$ which coincides with $\CM$ on $C_L\setminus\{\infty\}$ and whose stalk at $\infty$ is the sum $(\im\t^i+\ldots+\im\t^{i+l-1})_\infty$ inside $\CM(2k\cdot\infty)_\infty$. Then $\t$ defines homomorphisms $\t_i:\s\F_i\to\F_{i+1}$ for $0\le i<l$ because $\s\im\t^i=\im\s\t^i$ due to the flatness of $\s=(\id_C\times\sigma_L)^\ast$. Since $\CM_\infty\subset\CM(k\cdot\infty)_\infty=(\im\t^l)_\infty$ there are natural inclusions $\P_i:\F_i\to\F_{i+1}$ satisfying $\P_{i+1}\circ\t_i=\t_{i+1}\circ\s\P_i$ and $\im(\P_{l-1}\circ\ldots\circ\P_0)=\F_l(-k\cdot\infty)\subset\F_l$. We now set $\F_{i+nl}:=\F_i(kn\cdot\infty), \P_{i+nl}:=\P_i\otimes\id,\t_{i+nl}:=\t_i\otimes\id$ for $0\le i<l$ and $n\in\Z$. Clearly $\coker\t_i$ is supported on $\Graph(c)$ for all $i$ and isomorphic to $\coker\t$ which is an $L$-vector space of dimension $d$. We compute 
\[
\dim_L\coker\P_i\;=\;\deg\F_{i+1}-\deg\F_i\;=\;\deg\F_{i+1}-\deg\s\F_i\;=\;\dim_L\coker\t_i
\]
for all $i$. Hence $\FF=(\F_i,\P_i,\t_i)$ is an abelian $\tau$-sheaf over $L$ with $\ulM(\FF)=(M,\t)$.

\medskip
\noindent 
2. Let $l$ be an integer satisfying condition 2 of Definition~\ref{Def1.1}. For $0<i$ set 
\[
f_i:=\P'_{i-1}\circ\ldots\circ\P'_0\circ f_0\circ\P_0^{-1}\circ\ldots\circ\P_{i-1}^{-1}
\]
and similarly for $i<0$. Since the $\P_j,\P'_j$ are isomorphisms outside $\infty$ there exists an integer $m$ such that $f_i$ is a morphism $f_i:\F_i\to\F_i'(m\cdot\infty)$ for all $0\le i\le l$. Now the periodicity condition 2 of Definition~\ref{Def1.1} shows that the latter is a morphism for all $i\in\Z$. Finally $(\P'_0)^{-1}\t_0\,\s(f_0)=f_0\P_0^{-1}\t_0$ implies that $f=(f_i)_i:\FF\to\FF(m\cdot\infty)$ is a morphism with $\ulM(f)=f_0$ as desired.
\end{proof}

\medskip

The aforementioned relation can more generally be described by the following terminology. Let $\Spec\TA\subset C$ be an affine open subscheme.

\begin{Definition}\label{Def1.16}
A\/ \emph{$\t$-module on $\TA$ over $L$ of rank $r$} is a pair $\ulM=(M,\t)$, where
\begin{suchthat}
\item $M$ is a locally free $\TA\otimes_\Fq L$-module of rank $r$,
\item $\t: \s M\rightarrow M$ is injective.
\end{suchthat}
A\/ \emph{morphism} between $(M,\t)$ and $(M',\t')$ is a homomorphism $f:\, M\rightarrow M'$ of $\TA\otimes_\Fq L$-modules which respects $\t'\circ \s f = f\circ\t$. We denote the set of morphisms between $\ulM$ and $\ulM'$ by $\Hom(\ulM,\ulM')$.
\end{Definition}

Let $\FF$ be an abelian $\tau$-sheaf. 
Consider a finite closed subset $D\subset C$ such that either $\infty\in D$ or there exists a uniformizing parameter $z$ at infinity inside $\TA:=\Gamma(C\setminus D,\O_C)$. Note that by enlarging $D$ it will always be possible to find such a $z\in\TA$ in the case $\infty\not\in D$. 

If $\infty\in D$ we have by the $\P$'s a chain of isomorphisms
\[
\bigexact{\rule[-1.5ex]{0pt}{5ex}\cdots}%
{\sim\quad\;\;\;}
{\Gamma(C_L\setminus D,\F_{-1})}{\sim\quad\;\;\;}
{\Gamma(C_L\setminus D,\F_{ 0})}{\sim\quad\;\;\;}
{\Gamma(C_L\setminus D,\F_{ 1})}{\sim\quad\;\;\;}
{\cdots}
\]
since $\coker\P_i$ is only supported at $\infty$ for all $i\in\Z$. So we set $M:=\Gamma(C_L\setminus D,\F_{0})$ and $\t:=(\P_0|_{C_L\setminus D})^{-1}\circ\t_0|_{C_L\setminus D}$, and we define $\ulM^{(D)}(\FF):=(M,\t)$. Obviously, $\ulM^{(D)}(\FF)$ is a $\t$-module on $\TA$ and $\ulM^{(\infty)}(\FF)$ is the pure $A$-motive $\ulM(\FF)$ studied above.

If $\infty\notin D$ fix $z$ as above. Set $M_i:=\Gamma(C_L\setminus D,\F_i)$ and define
\begin{equation} \label{EQ.Tau}
\ulM^{(D)}(\FF) \,:=\, M_0\oplus\dots\oplus M_{l-1}\qquad\text{with} 
\qquad 
\t \,:=\, \left(
\raisebox{5ex}{$
\xymatrix @=0pc {
0 \ar@{.}[rrr] \ar@{.}[ddddrrrr] & & & 0 & **{!L !<0.8pc,0pc> =<1.5pc,1.5pc>}\objectbox{\TP^{-1}\circ z^k\t_{l-1}} \\
\t_0 \ar@{.}[dddrrr] & & & & 0 \ar@{.}[ddd]\\
0 \ar@{.}[dd] \ar@{.}[ddrr] \\
\\
0 \ar@{.}[rr] & & 0 & \t_{l-2} & 0
}$}
\qquad\qquad\right)
\end{equation}
and $\TP:=\P_{l-1}\circ\dots\circ\P_0$. Clearly $\t$ depends on the choice of $k$, $l$, and $z$. Again $\ulM^{(D)}(\FF)$ is a $\t$-module on $\TA:=\Gamma(C\setminus D,\O_C)$. Notice that $\coker\t$ is supported on $\Graph(c)\cap(C_L\setminus D)$.

\begin{Definition}\label{Def1.17}
We call\/ $\ulM^{(D)}(\FF)$ the \emph{$\t$-module on $\TA$ associated to $\FF$}. We abbreviate $\ulM^{(\{\infty\})}(\FF)$ to $\ulM(\FF)$.
\end{Definition}

If $\infty\notin D$ the $\t$-module $\ulM^{(D)}(\FF)$ is equipped with the endomorphisms
\begin{equation} \label{EQ.Pi+Lambda}
\P \,:=\, \left(
\raisebox{5.5ex}{$
\xymatrix @=0pc {
0 \ar@{.}[rrr] \ar@{.}[ddddrrrr] & & & 0 & **{!L !<0.8pc,0pc> =<1.5pc,1.5pc>}\objectbox{\TP^{-1}\circ z^k\P_{l-1}} \\
\P_0 \ar@{.}[dddrrr] & & & & 0 \ar@{.}[ddd]\\
0 \ar@{.}[dd] \ar@{.}[ddrr] \\
\\
0 \ar@{.}[rr] & & 0 & \P_{l-2} & 0
}$}
\qquad\qquad\right) , \enspace
\Lambda(\lambda) \,:=\, \left(
\raisebox{5.5ex}{$
\xymatrix @=0pc {
\!\!\lambda \cdot\id_{M_0}\!\!  \\
& \!\!\lambda^q \cdot\id_{M_1}\!\! \ar@{.}[ddrr] \\ 
\\
& & & \!\lambda^{q^{l-1}}\cdot\id_{M_{l-1}}\!\!
}$}
\right)
\end{equation}
for all $\lambda\in \Ff_{q^l}\cap L$. Actually $\Lambda(\lambda)$ is even an automorphism and the same holds for $\P$ if $z$ has no zeroes on $C\setminus(D\cup\{\infty\})$. They satisfy the relations $\P^l=z^k$ and $\P\circ\Lambda(\lambda^q)=\Lambda(\lambda)\circ\P$.

\begin{Lemma}\label{Lemma2.9a}
Assume that $\chr\notin D$ or that $\chr\ne\infty\in D$. If $\FF$ and $\FF'$ are abelian $\tau$-sheaves of different weights, then $\Hom(\ulM^{(D)}(\FF),\ulM^{(D)}(\FF'))=\{0\}$ (for any choice of $k$, $l$, $k'$, $l'$ and $z$ if $\infty\notin D$).
\end{Lemma}

Before proving the lemma we note a direct consequence of its interaction with Theorem~\ref{PropX.1}.

\begin{Corollary}\label{Cor2.9b}
If $\ulM$ and $\ulM'$ are pure $A$-motives of different weights, then $\Hom(\ulM,\ulM')=\{0\}$.
\qed
\end{Corollary}

\noindent
{\it Remark.}
Again this follows alternatively from the Dieudonn\'e-Manin classification of the local isoshtuka at $\infty$ associated with $\ulM,\ulM'$; see Section~\ref{SectLS}.

\begin{proof}[Proof of lemma~\ref{Lemma2.9a}]
Let $f\in \Hom(\ulM^{(D)}(\FF),\ulM^{(D)}(\FF'))$. If $\infty\in D$ set $\CM:=\F_0$ and $\CM':=\F'_0$. If $\infty\notin D$ set $\CM:=\bigoplus_{i=0}^{l-1}\F_i$ and $\CM':=\bigoplus_{i=0}^{l'-1}\F'_i$. Then $\Gamma(C_L\setminus D,\CM)=\ulM^{(D)}(\FF)$ and likewise for $\FF'$. Thus $f$ extends to a homomorphism $f:\CM\to\CM'(m\cdot D)$ for a suitable  $m\in\N$. We abbreviate $\t^i:=\t\circ\s(\t)\circ\ldots\circ(\s)^{i-1}(\t)$. Let $z\in Q$ be a uniformizing parameter at $\infty$. If $\infty\in D$ and $\chr\ne\infty$ then
\[
z^k\t^l:(\s)^l\CM_\infty \isoto\CM_\infty\qquad\text{and}\quad z^{k'}(\t')^{l'}:(\s)^{l'}\CM'_\infty \isoto\CM'_\infty
\]
are isomorphisms on the stalks at $\infty$. So for any $n\in\N$ we have for the stalk of $f$ at $\infty$
\[
f_\infty\;=\;\bigl(z^{k'}(\t')^{l'}\bigr)^{nl}\circ(\s)^{nll'}(z^{n(kl'-k'l)}f_\infty)\circ(z^k\t^l)^{-nl'}:\es\CM_\infty\to\CM'(m\cdot D)_\infty\,.
\]
In particular if $\frac{k}{l}>\frac{k'}{l'}$ (and similarly for $\frac{k}{l}<\frac{k'}{l'}$), $f_\infty\equiv0\mod z^{n(kl'-k'l)}$ for all $n$, whence $f_\infty=0$. Thus $f=0$ since $\CM$ is locally free.

If $\infty\notin D$ and $\chr\notin D$ then with the $\tau$ from (\ref{EQ.Tau}) the homomorphisms on the stalks at every point $v\in D$
\[
z^{-k}\t^l:(\s)^l\CM_v \isoto\CM_v\qquad\text{and}\quad z^{-k'}(\t')^{l'}:(\s)^{l'}\CM'_v \isoto\CM'_v
\]
are isomorphisms since $\chr\notin D$. So again for any $n\in \N$ the stalk $f_v$ satisfies
\[
f_v\;=\;\bigl(z^{-k'}(\t')^{l'}\bigr)^{nl}\circ(\s)^{nll'}(z^{-n(kl'-k'l)}f_v)\circ(z^{-k}\t^l)^{-nl'}:\es\CM_v\to\CM'(m\cdot D)_v\,.
\]
There exists a pole $v\in D$ of $z$. Then for $\frac{k}{l}>\frac{k'}{l'}$ (and similarly for $\frac{k}{l}<\frac{k'}{l'}$), $f_v=0$, whence $f=0$ as desired.
\end{proof}

\begin{Example}\label{Ex1.18b}
We give an example showing that the assertion of the lemma is false in case $\chr=\infty\in D$. Let $C=\PP^1_\Fq$, $\F_i=\O_{C_L}(i\cdot\infty)$, $\F'_i=\O_{C_L}(2i\cdot\infty)$ and let $\P_i$ and $\t_i$ be the natural inclusions $\F_i\subset\F_{i+1}$ and $\s\F_i\subset\F_{i+1}$ and likewise for $\F'_i$. Then $\FF=(\F_i,\P_i,\t_i)$ is an abelian $\tau$-sheaf of weight $1$ and $\FF'=(\F'_i,\P'_i,\t'_i)$ is an abelian $\tau$-sheaf of weight $2$ both of characteristic $\infty$. Clearly $\ulM^{(\infty)}(\FF)=(A_L,\id)=\ulM^{(\infty)}(\FF')$ contradicting the assertion of the lemma.
\end{Example}


\bigskip

\section{Kernel sheaf and image sheaf}

In this section we show that the kernel and the image of a morphism of pure $A$-motives are themselves pure $A$-motives and likewise for abelian $\tau$-sheaves provided that the characteristic point $\chr=c(\Spec L)$ is different from $\infty$. 

\begin{Proposition}\label{Prop1.9a}
Let $\ulM$ and $\ulM'$ be pure $A$-motives and let $f\in\Hom(\ulM,\ulM')$. Then
\[
\ker f \;:=\;(\ker f,\tau|_{\s\ker f})\qquad\text{and}\qquad \im f\;:=\;(\im f,\t'|_{\s\im f})
\]
are again pure $A$-motives with $\weight(\ker f)=\weight(\im f)=\weight(\ulM)$.
\end{Proposition}

\begin{proof}
Let $\CM,\CM',k,l,k',l',z$ be as in definition~\ref{Def1'.1} and the subsequent remark. After replacing $\CM'_\infty$ by $z^{-n}\CM'_\infty$ for an integer $n$ we may assume that $f$ extends to a morphism $\CM\to\CM'$. Since all local rings of $C_L$ are discrete valuation rings the subsheaves $\wt\CM:=\ker f$ and $\wh\CM:=\im f$ are themselves locally free by the elementary divisor theorem. Set $\wt M:=\Gamma(C_L\setminus\{\infty\},\wt\CM)$ and $\wh M:=\Gamma(C_L\setminus\{\infty\},\wh\CM)$. Moreover the restrictions $\tilde\t:=\t|_{\s\wt M}$ and $\hat\t:=\t'|_{\s\wh M}$ are clearly injective. If $f\ne0$ then $\weight(\ulM)=\weight(\ulM')$ by corollary~\ref{Cor2.9b}. Let $\tilde l=\hat l$ be the least common multiple of $l$ and $l'$ and let $\tilde k=\hat k=\weight(\ulM)\cdot\tilde l$. Then
\[
z^{\tilde k}\t^{\tilde l}:(\s)^{\tilde l}\CM_\infty\isoto\CM_\infty \qquad\text{and}\qquad
z^{\tilde k}(\t')^{\tilde l}:(\s)^{\tilde l}\CM'_\infty\isoto\CM'_\infty
\]
are isomorphisms. Thus also
\[
z^{\tilde k}\tilde\t^{\tilde l}:(\s)^{\tilde l}\wt\CM_\infty\isoto\wt\CM_\infty \qquad\text{and}\qquad
z^{\hat k}\hat\t^{\hat l}:(\s)^{\hat l}\wh\CM_\infty\isoto\wh\CM_\infty
\]
are isomorphisms. Since $\tau$ and $\tau'$ are isomorphism outside $\Graph(c)$ the same is true for $\tau|_{\s\ker f}$ and $\tau'|_{\s\im f}$. So the cokernels of the latter two are supported on $\Graph(c)$. This proves the proposition by Remark~\ref{Rem2.2}.
\end{proof}

\begin{Proposition}\label{KERISABELIAN}\label{IMISABELIAN}%
Let $\FF$ and $\FF'$ be abelian $\tau$-sheaves of characteristic different from $\infty$ and let $f\in\Hom(\FF,\FF')$. Then the \emph{kernel $\tau$-sheaf} and the \emph{image $\tau$-sheaf}
\[
\begin{array}{r@{\;}c@{\;}l}
\ker f &:=& (\ker f_i, \P_i|_{\ker f_i}, \t_i|_{\s\ker f_i}) \\[1ex]
\im f &:=& (\im f_i, \P'_i|_{\im f_i}, \t'_i|_{\s\im f_i})
\end{array}
\]
are abelian $\tau$-sheaves of the same characteristic as $\FF$ and $\FF'$.
\end{Proposition}

\begin{proof}\def\kerf{\ker f}\def\imf{\im f}%
We will conduct the proof for $\ker f$ and $\im f$ simultaneously. If $f=0$, then $\ker f=\FF$ and $\im f=\ZZ$, and we are done. Otherwise, we have a non-zero morphism between $\FF$ and $\FF'$, and by proposition \ref{PROP.1} both abelian $\tau$-sheaves $\FF$ and $\FF'$ have the same weight. We choose an integer $l$ that satisfies condition 2 of \ref{Def1.1} for both $\FF$ and $\FF'$ and we set $k=l\cdot\weight(\FF)$.

\def\kerf{\TF}\def\imf{\HF}%
Let $i\in\Z$. Since all local rings of $C_L$ are principal ideal domains the elementary divisor theorem yields that $\TF_i:=\ker f_i\subset\F_i$ and $\HF_i:=\im f_i\subset\F'_i$ are locally free coherent sheaves. 
The induced morphisms $\TP_i := \P_i|_{\kerf_i}$ and $\Tt_i := \t_i|_{\s\kerf_i}$ map injectively into $\kerf_{i+1}$ since $\s\ker f_i=\ker \s f_i$ due to the flatness of $\sigma$.
Similarly, we get this for $\HP_i := \P'_i|_{\imf_i}$ and $\Ht_i := \t'_i|_{\s\imf_i}$.
To examine $\coker\TP_i$ and $\coker\HP_i$ consider the diagram with exact
rows and columns in which the last column is exact by the 9-lemma
\[
\xymatrix{
& 0 \ar[d] & 0 \ar[d] & 0 \ar[d] & \\
0 \ar[r]& \kerf_i \ar[r]^{\TP_i}\ar[d]      & \kerf_{i+1} \ar[r]\ar[d]          & \coker\TP_i \ar[r]\ar[d]& 0 \\
0 \ar[r]& \F_i    \ar[r]^{\P_i} \ar[d]^{f_i}& \F_{i+1}    \ar[r]\ar[d]^{f_{i+1}}& \coker\P_i  \ar[r]\ar[d]& 0 \\
0 \ar[r]& \imf_i  \ar[r]^{\HP_i}\ar[d]      & \imf_{i+1}  \ar[r]\ar[d]          & \coker\HP_i \ar[r]\ar[d]& 0 \\
& 0 & 0 & 0 & \\
}
\]
Thus $\coker\TP_i$ and $\coker\HP_i$ are torsion sheaves like $\coker\P_i$, and we conclude that
the ranks $\Tr:= \rank\kerf_i$ and $\Hr:=\rank\imf_i$ are independent of $i$.

To show that $\underline\kerf$ and $\underline\imf$ are abelian $\tau$-sheaves let $\P$ and $\P'$ be the identifying morphisms $\P_{i+l-1}\circ\cdots\circ\P_i: \F_i \isoto \F_{i+l}(-k\cdot\infty)$ and $\P'_{i+l-1}\circ\cdots\circ\P'_i: \F'_i \isoto\F'_{i+l}(-k\cdot\infty)$, respectively. Since $\P$ and $\P'$ are isomorphisms we obtain the same for $\TP_{i+l-1}\circ\ldots\circ\TP_i$ and $\HP_{i+l-1}\circ\ldots\circ\HP_i$, whence the periodicity condition 2.

To establish conditions 3 and 4 we need that the characteristic is different from $\infty$. Let $c: \Spec L\rightarrow C':=C\setminus\{\infty\}$ and let
\[
\xymatrix{ & & &
\makebox[0pt][r]{$M \::=\:$ }%
\kerf_0|_{\CLa}\: \ar[r]^<<{\sim}^{\:\:\TP_0} &
\:\kerf_1|_{\CLa}\: \ar[r]^<<{\sim}^{\TP_1} & 
\:\cdots\makebox[2.4em][l]{}
}
\]
Set $\Tt := {\TP_0}^{-1}\!\circ\Tt_0:\; \s M\rightarrow M$ and set $\Td := \dim_L\coker\Tt$. Similar to the diagram chase for the $\coker\TP_i$, we get an injective morphism $\coker\Tt_i\rightarrow\coker\t_i$. Hence the support of $\coker\Tt_i$ lies outside $\infty$, and we have $\coker\Tt_i \:=\: \coker\Tt_i|_{\CLa} \:\cong\: \coker\Tt$ for all $i\in\Z$.
Now the exact sequences
\[
\begin{CD} 
{0} @>>> {\kerf_i}   @>{\TP_i}>> {\kerf_{i+1}} @>>> {\coker\TP_i} @>>> {0} \\[1ex]
{0} @>>> {\s\kerf_i} @>{\Tt_i}>> {\kerf_{i+1}} @>>> {\coker\Tt_i} @>>> {0} 
\end{CD}
\]
yield 
\[
\dim_L\coker\TP_i\:=\:\deg\kerf_{i+1}- \deg\kerf_i\:=\:\deg\kerf_{i+1}-\deg\s\kerf_i \:=\:\dim_L\coker \Tt_i\:=\:\dim_L\coker\Tt \:=\: \Td
\]
for all $i\in\Z$.
Clearly, $\coker\Tt_i$ is supported on the graph of $c$ due to its injection into $\coker\t_i$. Again, this argument adapts to $\coker\HP_i$ and $\coker\Ht_i$, as well.
\end{proof}

\begin{Corollary}
Let\/ $\FF$ and $\FF'$ be abelian $\tau$-sheaves of characteristic different from $\infty$ and let\/ $f \in \Hom(\FF,\FF')$ be a morphism. 
\begin{suchthat}
\item $f$ is injective if and only if\/ $\ker f=\ZZ$\,.
\item $f$ is surjective if and only if\/ $\im f=\FF'$\,. \qed
\end{suchthat}
\end{Corollary}

\begin{Example} \label{ExTRichter}
As was pointed out to us by T.\ Richter the assumption $\chr\ne\infty$ cannot be dropped. For instance consider the abelian $\tau$-sheaf on $C_L=\PP^1_L$ with $\F_i=\O_{\PP^1_L}\bigl(\bigl\lceil\frac{i-1}{2}\bigr\rceil\bigr)\oplus\O_{\PP^1_L}\bigl(\bigl\lceil\frac{i}{2}\bigr\rceil\bigr)$,
where $\bigl\lceil\frac{i}{2}\bigr\rceil$ denotes the smallest integer $\ge\frac{i}{2}$. Let $\P_i=\left(\begin{array}{cc} 1 & 0 \\ 0 & 1 \end{array}\right)$ and $\t_i=\left(\begin{array}{cc} 0 & 1 \\ 1 & 0 \end{array}\right)$. Then $\FF=(\F_i,\P_i,\t_i)$ is an abelian $\tau$-sheaf with $r=l=2,d=k=1$, and characteristic $\infty$. We rewrite everything in terms of the bases $\biggl(\!\!\!\begin{array}{c}z^{-\lceil\frac{i-1}{2}\rceil}\\ 0 \end{array}\!\!\!\biggr),\biggl(\!\!\begin{array}{c}0\\ z^{-\lceil\frac{i}{2}\rceil}\end{array}\!\!\biggr)$ of $\F_i|_{\PP^1_L\setminus\{0\}}$, where $\PP^1_L\setminus\{0\}=\Spec L[z]$. With respect to these bases $\P_i$ and $\t_i$ are described by the matrices
\[
\P_i=\left(\begin{array}{cc} 1 & 0 \\ 0 & z \end{array}\right)\text{ for }2|i\,,\quad\P_i=\left(\begin{array}{cc} z & 0 \\ 0 & 1 \end{array}\right)\text{ for }2\nmid i\,,\qquad \text{and} \qquad \t_i=\left(\begin{array}{cc} 0 & 1 \\ z & 0 \end{array}\right)\text{ for all }i\,.
\]
There is an endomorphism $f$ of $\FF$ given by $f_i=\left(\begin{array}{cc} z & z \\ z & z \end{array}\right)\text{ for }2|i\,$ and $f_i=\left(\begin{array}{cc} z & 1 \\ z^2 & z \end{array}\right)\text{ for }2\nmid i$. We compute
\[
\begin{array}{rcll}
\TS\ker f_i&=&\TS{-1\choose 1}\cdot\O_{\PP^1_L}(\frac{i}{2}\cdot\infty) & \quad \text{ if }2|i \text{ and}\\[2mm]
\TS\ker f_i&=&\TS{-1\choose z}\cdot\O_{\PP^1_L}(\frac{i-1}{2}\cdot\infty) & \quad \text{ if }2\nmid i\,.
\end{array}
\]
Therefore $\P_i|_{\ker f_i}$ is an isomorphism if $2|i$ and has cokernel of $L$-dimension $1$ for $2\nmid i$. Thus $\ker f$ is not an abelian $\tau$-sheaf.
\end{Example}


\bigskip

\section{Isogenies between abelian $\tau$-sheaves and pure $A$-motives}

In the theory of abelian varieties the concept of \emph{isogenies} is central, defining an equivalence relation which allows a classification of abelian varieties into isogeny classes that are larger than isomorphism classes. In the following, we adapt the idea of isogenies to abelian $\tau$-sheaves. They are defined by the following conditions.

\begin{Proposition}\label{PROP.1.42A}%
  Let $f:\FF\to\FF'$ be a morphism between two abelian $\tau$-sheaves $\FF = (\F_i,\P_i,\t_i)$ and\/ $\FF' = (\F'_i,\P'_i,\t'_i)$. Then the following assertions are equivalent:
\begin{suchthat}
\item 
$f$ is injective and the support of all\/ $\coker f_i$ is contained in $D\times\Spec L$ for a finite closed subscheme $D\subset C$,
\item 
$f$ is injective and $\FF$ and $\FF'$ have the same rank and dimension,
\item 
$\FF$ and $\FF'$ have the same weight and the fiber $f_{i,\eta}$ at the generic point $\eta$ of $C_L$ is an isomorphism for some (any) $i\in\Z$.
\end{suchthat}
\end{Proposition}

\begin{proof}
$1\Rightarrow 3$ follows from \ref{PROP.1} and the fact that $\P_{i,\eta}$ and $\P'_{i,\eta}$ are isomorphisms for all $i$.
Since $3\Rightarrow 2$ is evident it remains to establish $2\Rightarrow 1$.

We will first reduce to the case $A=\BF_q[t]$. By the theorem of Riemann-Roch there exists a rational function $t\in Q$ on $C$ with poles only at $\infty$ and whose zeroes do not meet the characteristic point, nor the support of the $\coker f_i$. This function defines an inclusion of function fields $\Fq(t)\subset Q$ and thus a finite flat morphism between the respective curves $\varphi:\, C\rightarrow \PP^1_\Fq$ with $\varphi^{-1}(\infty_{\PP^1})=\{\infty\}$. The direct images $\GG:=\varphi_\ast\FF$ and $\GG':=\varphi_\ast\FF'$ under $\varphi$ are abelian $\tau$-sheaves on $\PP^1_\Fq$ of rank $r\cdot\deg\varphi$, dimension $d$, and characteristic $\varphi\circ c$ by \cite[Proposition~1.6]{HH}. We define $\TA:=\Gamma(\PP^1_\Fq\setminus\{0\},\O_{\PP^1})$ such that $\TA=\Fq\II{z}$ for some $z\in\TA$ with a simple pole at $0$ and a simple zero at $\infty$. We choose an integer $l$ that satisfies condition 2 of \ref{Def1.1} for both $\GG$ and $\GG'$. Consider $\ulM^{(0)}(\GG)=(M,\t)$ and $\ulM^{(0)}(\GG')=(M',\t')$; see Definition~\ref{Def1.17}. Set $s:=lr\deg\varphi=\rk M$ and $e:=ld=s\cdot\weight(\GG)$.

Now choose $\TA_L$-bases of $M$ and $M'$. This is possible since $\TA_L$ is a principal ideal domain and that was the reason why we constructed $\varphi$. With respect to these bases, the endomorphisms $\t$ and $\t'$ and the $\TA$-morphism $g=\ulM^{(0)}(\varphi_\ast f):\, M\rightarrow M'$ which is induced by $f$ can be described by quadratic matrices $T$, $T'$ and $H$, and we have the formula $T'\s\!H = HT$.

Let $\zeta:=c^\ast(z)\in L$. By the elementary divisor theorem we find matrices $U,V\in GL_s(\AxL)$ with 
\[
UTV \,=\, \left(
\begin{array}{ccc} 
\makebox[2.5em][l]{$(z-\zeta)^{n_1}$} & & \\
& \ddots & \\
& & \makebox[2.5em][r]{$(z-\zeta)^{n_s}$}
\end{array}
\right)
\]
for some integers $n_1\le\dots\le n_s$. Thus $\coker\t\cong \bigoplus_{i=1}^s \AxL\,/\,(z-\zeta)^{n_i}$ and  $e=\sum_{i=1}^s n_i$. Since
\[
\det T\cdot\det UV = \det \,UTV = (z-\zeta)^e 
\]
we calculate $\det T=b\cdot(z-\zeta)^e$ for some $b\in(\AxL)^{\times}={L\II{z}}^{\!\times}=L^{\!\times}$. Analogously, we have $\det T'=b'\cdot(z-\zeta)^e$ for some $b'\in L^{\!\times}$ as $\GG$ and $\GG'$ have the same dimension $d$. Since $\det H\ne0$ due to the injectivity of $f$ we conclude
\[
\det T'\cdot\det\s\!H \,=\, \det H\cdot\det T 
\quad\Rightarrow\quad
\frac{\displaystyle \det\s\!H}{\displaystyle \det H} \,=\, \frac{\displaystyle b}{\displaystyle b'\!\!} \;\in L^{\!\times}\,.
\]
In an algebraic closure $L^\alg$ of $L$ there exists a $\lambda$ with $\lambda^{q-1}=\frac{b'\!\!}{b}$. So we have 
\[
a:=\lambda\cdot\det H=\s(\lambda\cdot\det H)\in L^\alg[z]
\]
and, due to the $\sigma$-invariance, even $a\in\Fq\II{z}=\TA$ (and hence $\lambda\in L$).
Again using the elementary divisor theorem one sees that $a$ annihilates $\coker g$. Now our proof is complete as the support of $\coker f_i$ is contained in the divisor of zeroes $(\varphi^\ast(a))_0\times\Spec L$ on $C_L$ for $0\le i<l$ by construction (for this purpose we used $g=\ulM^{(0)}(\varphi_\ast f)$ which captures all these $f_i$) and for the remaining $i\in\Z$ by periodicity.
\end{proof}

\begin{Definition}[Isogeny]
\begin{suchthat}
\item 
A morphism $f:\FF\to \FF'$ satisfying the equivalent conditions of proposition~\ref{PROP.1.42A} is called an \emph{isogeny}.
We denote the set of isogenies between $\FF$ and $\FF'$ by $\Isog(\FF,\FF')$.
\item 
An isogeny $f:\FF\to\FF'$ is called \emph{separable} (respectively \emph{purely inseparable}) if for all $i$ the induced morphism $\t_i:\s\coker f_i\to\coker f_{i+1}$ is an isomorphism (respectively is \emph{nilpotent}, that is, $\t_{i}\circ\s\t_{i-1}\circ\ldots\circ(\s)^n\t_{i-n}=0$ for some $n$).
\end{suchthat}
\end{Definition}

\medskip

Proposition~\ref{PROP.1.42A} has  important consequences for pure $A$-motives.

\begin{Corollary}\label{Cor1.26a}
Let $f:\FF\to\FF'$ be a morphism between abelian $\tau$-sheaves of characteristic different from $\infty$. Then $f$ is an isogeny if and only if $\ulM(f):\ulM(\FF)\to\ulM(\FF')$ is an isogeny between the associated pure $A$-motives.\qed
\end{Corollary}

\begin{Corollary}\label{Cor1.11b}
Let $f:\ulM\to\ulM'$ be an isogeny between pure $A$-motives (Definition~\ref{Def1'.2}). Then
\begin{suchthat}
\item
there exists an element $a\in A$ which annihilates $\coker f$,
\item 
there exists a dual isogeny $\dual{f}:\ulM'\to\ulM$ such that $f\circ\dual{f}=a\cdot\id_{\ulM'}$ and $\dual{f}\circ f=\id_\ulM$.
\end{suchthat}
\end{Corollary}

\begin{proof}
1 follows from Corollary~\ref{Cor2.9b}, Theorem~\ref{PropX.1}, and Proposition~\ref{PROP.1.42A} by noting that $D$ is contained in the zero locus of a suitable $a\in A$ by the Riemann-Roch theorem.

For 2  consider the diagram 
\[
\xymatrix{
0 \ar[r] &
M \ar[r]^{f}\ar[d]_{a} &
M' \ar[r]\ar[d]_{a} \ar@{-->}[dl]_{\dual{f}} &
\coker f \ar[d]_{a}^{\;(=0)} \ar[r] & 
0
\\
0 \ar[r] &
M \ar[r]^{f} &
M' \ar[r] &
\coker f \ar[r] & 
0\,.
}
\]
By diagram chase, we get a morphism $\dual{f}: M'\rightarrow M$ which is \emph{dual} to $f$ in the sense that $\dual{f}\circ f = a$ and $f\circ\dual{f} = a$. 
\end{proof}

\begin{Remark}\label{Rem1.26'}
The dual isogeny $\dual{f}$ clearly depends on $a$ and there rarely is a canonical choice for $a$. If $C=\PP^1$ and $A=\Fq[t]$ we obtain from the elementary divisor theorem a unique minimal monic element $a\in A$ (which still depends on the choice of the isomorphism $A\cong\Fq[t]$, though) that annihilates $\coker f$. Also if $f\in\End(\ulM)$ is an isogeny of a semisimple pure $A$-motive over a finite field we will exhibit in \cite[Theorem~\BHBPropAAA]{BH_B} a canonical $a$. But in general there is no canonical choice.

Nevertheless, since $A$ is a Dedekind domain, a power of the ideal annihilating $\coker f$ will be principal and one may take $a$ as a generator. This has the advantage that the support of $\coker f$ equals $V(a)\subset\Spec A$. In particular if the characteristic point $\chr$ is not contained in the support of $\coker f$ and in $V(a)$, also $\dual{f}$ will be separable. On the other hand, if $f\in \End(\ulM)$ then $f$ is integral over $A$, since $\End(\ulM)$ is a finite $A$-module by Proposition~\ref{PropT.1} below. Then $f$ generates a finite commutative $A$-algebra $A[f]$. Our discussion of the choice of $a$ shows that the set $V(f)\subset\Spec A[f]$ of zeroes of $f$ lies above $\supp(\coker f)\subset\Spec A$.
\end{Remark}

\begin{Lemma}\label{ISOGENYCOMPOSE}
Let $f\in\Hom(\FF,\FF')$ and $f'\in\Hom(\FF',\FF'')$ be morphisms between abelian $\tau$-sheaves and let $D$ be a divisor on $C$. 
\begin{suchthat}
\item If two of $f$, $f'$, and $f'\circ f$ are isogenies then so is the third.
\item If $f$ is an isogeny then also $f\otimesidOCL{D}\in\Isog(\FF(D),\FF'(D))$ is an isogeny.
\item If $D$ is effective then the natural inclusion $\FF\subset\FF(D)$ is an isogeny.
\end{suchthat}
\end{Lemma}

\begin{proof}
1 is obvious.

\smallskip
\noindent 
2. Clearly the tensored morphisms $f_i\otimes 1:\F_i(D)\rightarrow\F'_i(D)$ remain injective and the support of $\coker f_i\otimes 1$ equals the support of $\coker f_i$.

\smallskip
\noindent
3. The inclusion $\FF\subset\FF(D)$ is a morphism because the divisor $D$ is $\sigma$-invariant.
\end{proof}


\bigskip

\section{Quasi-morphisms and  quasi-isogenies}

We want to establish the existence of dual isogenies also for abelian $\tau$-sheaves.
But if we follow the proof of Corollary~\ref{Cor1.11b}, the problem is that multiplication with $a$ is not an endomorphism of an abelian $\tau$-sheaf, since it produces poles. We remedy this by defining \emph{quasi-morphisms} and \emph{quasi-isogenies} between $\FF$ and $\FF'$ which allow the maps to have finite sets of poles.

\begin{Definition}[Quasi-morphism and quasi-isogeny]
Let $\FF$ and $\FF'$ be abelian $\tau$-sheaves.
\begin{suchthat}
\item A \emph{quasi-morphism} $f$ between $\FF$ and $\FF'$ is a morphism $f\in\Hom(\FF,\FF'(D))$ for some effective divisor $D$ on $C$.
\item A \emph{quasi-isogeny} $f$ between $\FF$ and $\FF'$ is an isogeny $f\in\Isog(\FF,\FF'(D))$ for some effective divisor $D$ on $C$. 
\end{suchthat}
We call two quasi-morphisms $f_1\in\Hom(\FF,\FF'(D_1))$ and $f_2\in\Hom(\FF,\FF'(D_2))$ \emph{equivalent} $(\text{denoted }f_1\sim f_2)$, if the diagram
\[
\xymatrix@=0pt@C=3em{ 
& \FF'(D_1) \ar[rd] & \\
\FF \ar[ru]^{f_1}\ar[rd]_{f_2} & & \FF'\makebox[0pt][l]{$(D_1\!+\!D_2)$} \\
& \FF'(D_2) \ar[ru] &
}
\]
commutes where the two arrows on the right are the natural inclusions.
\end{Definition}

Clearly, the relation $\sim$ defines an equivalence relation on the set of quasi-morphisms between $\FF$ and $\FF'$ where the transitivity is seen from
\[
\xymatrix@=3ex@C=3em{ 
& \FF'(D_1) \ar[r]\ar@{-->}[rd] & \FF'\makebox[0pt][l]{$(D_1\!+\!D_2)$} &\quad\; \ar[rd] & \\
\FF \ar[ru]^{f_1}\ar[r]^{f_2}\ar[rd]_{f_3} & \FF'(D_2) \ar[ru]\ar[rd] & \FF'\makebox[0pt][l]{$(D_1\!+\!D_3)$} &\; \ar@{.>}[r] & \FF'\makebox[0pt][l]{$(D_1\!+\!D_2\!+\!D_3)$} \\
& \FF'(D_3) \ar[r]\ar@{-->}[ru] & \FF'\makebox[0pt][l]{$(D_2\!+\!D_3)$} &\quad\; \ar[ru] &
}
\]
by canceling the dotted arrow due to injectivity. Since the divisors of quasi-morphisms are not particularly interesting, we fade them out by forming equivalence classes of quasi-morphisms according to this equivalence relation.

\medskip

\begin{Definition}
Let $\FF$ and $\FF'$ be abelian $\tau$-sheaves. \nopagebreak
\begin{suchthat}
\nopagebreak
\item We set $\QHom(\FF,\FF')$ to be the set of quasi-morphisms between $\FF$ and $\FF'$ modulo $\sim$. 
\item The equivalence class of a quasi-morphism $f$ between $\FF$ and $\FF'$ modulo $\sim$ is denoted by $\II{f}$, and we call it a \emph{quasi-morphism} between $\FF$ and $\FF'$ as well.
\item We set $\QIsog(\FF,\FF')$ to be the subset of\/ $\QHom(\FF,\FF')$ whose elements $\II{f}$ are represented by isogenies $f$.
\end{suchthat}
\vspace{1mm}
We write $\QEnd(\FF):=\QHom(\FF,\FF)$ and $\QIsog(\FF):=\QIsog(\FF,\FF)$.
\end{Definition}

\noindent {\it Remark.}\INDENT
1. By Lemma~\ref{ISOGENYCOMPOSE}, it holds for $f_1\sim f_2$, that $f_1$ is a quasi-isogeny if and only if $f_2$ is a quasi-isogeny. This justifies our definition of $\QIsog(\FF,\FF')$.

2. Proposition \ref{PROP.1.42A} and Lemma~\ref{ISOGENYCOMPOSE} hold analogously for quasi-morphisms and quasi-isogenies, since every quasi-morphism $f\in\QHom(\FF,\FF')$ is a morphism $f\in\Hom(\FF,\FF'(D))$ for some effective divisor $D$ on $C$.

3. Every pair of quasi-morphisms $\II{f_1},\II{f_2}\in\QHom(\FF,\FF')$ can be represented by morphisms $f_1,f_2\in\Hom(\FF,\FF'(D))$ with the same divisor $D$ on $C$. 
In particular we can form the sum
\[
\II{f_1}+\II{f_2} := \II{f_1+f_2} \:\in\QHom(\FF,\FF')\,.
\]
Since poles are negligible, we can extend this structure to a $Q$-vector space by now being able to admit multiplication by elements of $Q$. Let $\II{f}\in\QHom(\FF,\FF')$ be represented by $f\in\Hom(\FF,\FF'(D))$ and let $a\in Q$. Then $a\cdot f\in\Hom(\FF,\FF'(D\!+\!(a)_\infty))$ where $(a)_\infty$ denotes the divisor of poles of $a$, and we define
\[
a\cdot\II{f} := \II{a\cdot f} \:\in\QHom(\FF,\FF')\,.
\]
Moreover, Quasi-morphisms can be composed. Let $\FF$, $\FF'$ and $\FF''$ be abelian $\tau$-sheaves and let $\II{f}\in\QHom(\FF,\FF')$ and $\II{f'}\in\QHom(\FF',\FF'')$ be quasi-morphisms between $\FF$, $\FF'$ and $\FF''$, which are represented by $f\in\Hom(\FF,\FF'(D))$ and $f'\in\Hom(\FF',\FF''(D'))$, respectively. In order to compose $f'$ and $f$, we have to raise $f'$ to be a morphism from $\FF'(D)$. We achieve this by simply tensoring with $\otimes_\OCL\OCL(D)$. Now $(f'\otimesidOCL{D})\circ f\in\Hom(\FF,\FF''(D\!+\!D'))$, and we can define the composition
\[
\II{f'}\circ\II{f} := \II{(f'\otimesidOCL{D})\circ f} \:\in\QHom(\FF,\FF'')\,.
\]
All these operations are well-defined which can easily be seen by diagram arguments similar to the one we presented for the transitivity of $\sim$. Altogether we obtain

\begin{Corollary}\label{QEND-Q-ALGEBRA}\label{COMPOSITIONISISOGENY}
Let $\FF$ and $\FF'$ be abelian $\tau$-sheaves. With the above given structure, we have
\begin{suchthat}
\item the composition of quasi-isogenies is again a quasi-isogeny,
\item $\QHom(\FF,\FF')$ is a $Q$-vector space,
\item $\QEnd(\FF)$ is a $Q$-algebra. \qed
\end{suchthat}
\end{Corollary}

\smallskip

\noindent {\it Remark.}
The $Q$-vector spaces $\QHom(\FF,\FF')$ and $\QEnd(\FF)$ are finite dimensional. We will prove this in Proposition~\ref{PropT.2} below.

\bigskip

As an abuse of notation, we will write $f\in\QHom(\FF,\FF')$ instead of $\II{f}$ to denote the quasi-morphism represented by $f\in\Hom(\FF,\FF'(D))$.

\begin{Remark} \label{MULTISISOGENY}
For every $a\in Q\mal$, the multiplication by $a$ is a quasi-isogeny on $\FF$. Since $a$ injects $\F_i$ into $\F_i((a)_\infty)$ and commutes with the $\P_i$ and the $\t_i$, it is a morphism of abelian $\tau$-sheaves. Additionally, its cokernels are supported on $(a)_0$, the divisor of zeroes of $a$.
\end{Remark}

Now we come back to the idea of defining a dual isogeny. As already mentioned, a global definition fails because the annihilating multiplication by $a$ is not a morphism between $\F_i$ and $\F'_i$. This problem will now be solved by using quasi-morphisms and quasi-isogenies.

Let $\FF$ and $\FF'$ be abelian $\tau$-sheaves and let $f\in\QIsog(\FF,\FF')$ be a quasi-isogeny represented by an isogeny $f:\FF\to\FF'(D)$ for an effective divisor $D$ on $C$. By the annihilating property of the support, we can find $a\in Q\mal$ with $a\cdot\coker f_i=0$ for all $i\in\Z$. Now consider the following diagram.
\[
\xymatrix@M=0.5em{
0 \ar[r] &
\F_i \ar[rr]^{f_i}\ar[d]_{a} &&
\F'_i(D) \ar[rr]\ar[d]_{a} \ar@{-->}[dll]_{\dual{f_i}} &&
\coker f_i \ar[d]_{a}^{\;(=0)} \ar[r] & 
0
\\
0 \ar[r] &
\F_i\makebox[0pt][l]{$((a)_\infty)$} & \ar[r]^{f_i\quad} &
\F'_i\makebox[19pt][l]{$(D+(a)_\infty)$} & \ar[r] &
\coker f_i \ar[r] & 
0\,.
}
\]
As in \ref{Cor1.11b}, we get a morphism $\dual{f_i}: \F'_i(D)\rightarrow\F_i\bigl((a)_\infty\bigr)$ satisfying $\dual{f_i}\circ f_i = a$ and $f_i\circ\dual{f_i} = a$. Collecting these $\dual{f_i}$ together, we obtain a \emph{dual} morphism of abelian $\tau$-sheaves $\dual{f}\in\Hom(\FF'(D),\FF((a)_\infty))$ which is a quasi-morphism between $\FF'$ and $\FF$. 

\begin{Proposition} \label{QISOG-GROUP}%
Let $\FF$ and $\FF'$ be abelian $\tau$-sheaves.
\begin{suchthat}
\item
If $f\in \QIsog(\FF,\FF')$ is a quasi-isogeny then the dual $\dual{f}\in\QHom(\FF',\FF)$ of\/ $f$ is a quasi-isogeny and $f^{-1}:=a^{-1}\cdot\dual{f}$ is the inverse of $f$ in $\QHom(\FF',\FF)$.
\item 
$\QIsog(\FF)$ is the group of units in the $Q$-algebra $\QEnd(\FF)$.
\end{suchthat}
\end{Proposition}

\begin{proof}
Since the $f_i$ and the multiplication by $a\not=0$ are isomorphisms at the generic fiber the lemma follows from proposition~\ref{PROP.1.42A}.
\end{proof}

\noindent {\it Remark.}
The dual morphism $\dual{f}$ clearly depends on the choice of $a$ and again there is in general no canonical choice of $a$.

\bigskip

\begin{Definition}
Let $\FF$ and $\FF'$ be abelian $\tau$-sheaves. We call $\FF$ and $\FF'$ \emph{quasi-isogenous} $(\FF\approx\FF')$, if there exists a quasi-isogeny between $\FF$ and $\FF'$.
\end{Definition}

\begin{Corollary}\label{QISOG-EQREL}
The relation $\approx$ is an equivalence relation. \qed
\end{Corollary}

\begin{Proposition}\label{QHOM-QEND-ISOMORPHIC}
Let $\FF$ and $\FF'$ be abelian $\tau$-sheaves. If $\FF\approx\FF'$, then 
\begin{suchthat}
\item the $Q$-algebras $\QEnd(\FF)$ and $\QEnd(\FF')$ are isomorphic,
\item $\QHom(\FF,\FF')$ is free of rank $1$ both as a left module over $\QEnd(\FF')$ and as a right module over $\QEnd(\FF)$. \qed
\end{suchthat}
\end{Proposition}

Next we want to give an alternative description of $\QHom(\FF,\FF')$ similar to the description \cite[Proposition 3.4.5]{Papanikolas} in the case of ``dual $t$-motives''.

\begin{Proposition}\label{PropAltDescrQHom}
Let $\FF$ and $\FF'$ be abelian $\tau$-sheaves of the same weight and characteristic. Then the $Q$-vector space $\QHom(\FF,\FF')$ is canonically isomorphic to the space of morphisms between the fibers at the generic point $\eta$ of $C_L$
\[
\bigl\{\,f_{0,\eta}:\F_{0,\eta}\to\F'_{0,\eta}\enspace \text{with}\enspace f_{0,\eta}\circ\P_{0,\eta}^{-1}\circ\t_{0,\eta}\;=\;(\P'_{0,\eta})^{-1}\circ\t'_{0,\eta}\circ\sigma^\ast(f_{0,\eta})\,\bigr\}\,.
\]
This isomorphism is compatible with composition of quasi-morphisms.
\end{Proposition}

\begin{proof}
Clearly if $f\in \QHom(\FF,\FF')$ the map $f\mapsto f_{0,\eta}$ is a monomorphism of $Q$-vector spaces. To show that it is surjective let $f_{0,\eta}$ as above be given. As in the proof of \ref{PROP.1.42A} choose a finite flat morphism $\varphi:C\to\PP^1_{\Fq}$ with $\varphi^{-1}(\infty_{\PP^1})=\{\infty\}$, set $\Fq[t]=\Gamma(\PP^1_\Fq\setminus\{\infty_{\PP^1}\},\O_{\PP^1})$, and replace $\FF$ and $\FF'$ by $\varphi_\ast\FF$ and $\varphi_\ast\FF'$. Choose $L[t]$-bases of $M=\Gamma(\PP^1_L\setminus\{\infty_{\PP^1}\},\F_0)$ and $M'=\Gamma(\PP^1_L\setminus\{\infty_{\PP^1}\},\F'_0)$, and write $\P_0^{-1}\circ\t_0$ and $(\P'_0)^{-1}\circ\t'_0$ with respect to these bases as matrices $T$ and $T'$ with coefficients in $L[t]$. If $\chr\ne\infty$ let $\theta:=c^\ast(t)$ and set $e:=d$ and $e':=d'$. If $\chr=\infty$ set $e=e'=0$ and $\theta:=0$ (the choice of $\theta$ will not play a role in this case).

Then in both cases $\det T=b\cdot(t-\theta)^e$ and $\det T'=b'\cdot(t-\theta)^{e'}$ for $b,b'\in L^{\times}$. By considering the adjoint matrices we find in particular that $(t-\theta)^eT^{-1}$ and $(t-\theta)^{e'}(T')^{-1}$ have all their coefficients in $L[t]$. Write $f_{0,\eta}$ with respect to these bases as a matrix $F\in M_{r'\times r}(L(t))$. It satisfies $FT=T'\s F$.

Consider the ideals of $L^\alg[t]$ where $L^\alg$ is an algebraic closure of $L$
\[
I:=\bigl\{\,h\in L^\alg[t]:\enspace hF\in M_{r'\times r}(L^\alg[t])\,\bigr\}
\]
and $I^\sigma:=\{\s(h):h\in I\}$. Note that $I\ne(0)$. We claim that
\begin{eqnarray*}
h\in I \enspace& \Longrightarrow & (t-\theta)^{e'}h\in I^\sigma \quad\text{and}\\
h\in I^\sigma &\Longrightarrow & (t-\theta)^e h\in I\,.
\end{eqnarray*}
Indeed, let $h\in I$ and set $g:=(\s)^{-1}((t-\theta)^{e'}h)$. Then 
\[
\s(gF)\;=\;(t-\theta)^{e'}h\,\s F\;=\;(t-\theta)^{e'}(T')^{-1}\cdot hFT\;\in\; M_{r'\times r}(L^\alg[t])\,.
\]
Hence $g\in I$ and $(t-\theta)^{e'}h\in I^\sigma$. Conversely let $h\in I^\sigma$, that is, $h=\s(g)$ for $g\in I$. Then
\[
(t-\theta)^ehF\;=\;T'\s(gF)\cdot(t-\theta)^eT^{-1}\;\in\;M_{r'\times r}(L^\alg[t])
\]
proving the claim.

Since $L^\alg[t]$ is a principal ideal domain we find $I=(h)$ and $I^\sigma=(\s (h))$ for some $h\in I$. In particular $(t-\theta)^{e'}h=g\cdot\s(h)$ and $(t-\theta)^e \s(h)=f\cdot h$ for suitable $f,g\in L^\alg[t]$. We conclude $(t-\theta)^{e+e'}h=fg\,h$ and since the polynomials $h$ and $\s(h)$ are non-zero and have the same degree, $g=\beta\cdot(t-\theta)^{e'}$ for some $\beta\in (L^\alg)^{\times}$. Choose an element $\gamma\in (L^\alg)^{\times}$ with $\gamma^{q-1}=\beta$. Then
\[
a:=\gamma h=\s(\gamma h)\in \Fq[t]
\]
and $aF\in M_{r'\times r}(L[t])$. Thus $f_{0,\eta}$ defines a morphism $f_0:\F_0\to\F'_0(D_0)$ for $D_0:=(\varphi^\ast a)_0+m_0\cdot\infty$ with appropriate $m_0\in \N_0$. Here $(\varphi^\ast a)_0$ is the zero divisor of the element $\varphi^\ast a\in A$.

Now we define inductively on $C'_L:=C_L\setminus\{\infty\}$
\[
f^{\,}_i := \P'_{i-1}\circ f^{\,}_{i-1}\circ\P_{i-1}^{-1}:\; \F_i|_\CLa\rightarrow\F'_i((\varphi^\ast a)_0)|_\CLa \qquad (i>0)
\]
and analogously for $i<0$. To pass to the projective closure, we allow divisors $D_i=(\varphi^\ast a)_0+m_i\cdot\infty$ for sufficiently large $m_i>0$ such that $f_i:\,\F_i\rightarrow\F'_i(D_i)$ for all $i\in\Z$. Since $\FF$ and $\FF'$ have the same weight, we have the periodical identification if $l$ satisfies condition 2 of Definition~\ref{Def1.1} for both $\FF$ and $\FF'$
\[
\xymatrix{
& \F_{i+nl} \ar[r]^<<{\sim}\ar[d]^{f_{i+nl}} & \F_i(nk\cdot\infty) \ar[d]^{f_i\,\otimesidOCL{nk\cdot\infty}} & \ar@{}[d]^{\displaystyle(i,n\in\Z)\,.} \\
& \F'_{i+nl}(D_i) \ar[r]^<<{\sim} & \F'_i(D_i+nk\cdot\infty) &
}
\]
Take $m:=\max\{m_0,\dots,m_{l-1}\}$ and set $D:=(\varphi^\ast a)_0+m\cdot\infty$. Then $f_i:\,\F_i\rightarrow\F'_i(D)$ for all $i\in\Z$. Since the commutation with the $\P$'s and the $\t$'s holds by construction, the collection of the $f_i$ is the desired quasi-morphism $f\in\QHom(\FF,\FF')$. 
\end{proof}

\bigskip

The following proposition connects the theory of quasi-morphisms of abelian $\tau$-sheaves to the theory of morphisms of their associated pure $A$-motives and $\t$-modules.

\begin{Proposition}\label{CONNECTION}
Let $\FF$ and $\FF'$ be two abelian $\tau$-sheaves of the same weight and let $D\subset C$ be a finite closed subscheme as in Section~\ref{SectRelation}. 
\begin{suchthat}
\item
If $\infty\in D$ we have a canonical isomorphism of $Q$-vector spaces
\[
\QHom(\FF,\FF') \;=\; \Hom(\ulM^{(D)}(\FF),\ulM^{(D)}(\FF'))\otimes_{\TA} Q\,.
\]
\item
If $\infty\notin D$ choose an integer $l$ which satisfies condition 2 of \ref{Def1.1} for both $\FF$ and $\FF'$ and assume $\Ff_{q^l}\subset L$. Then we have a canonical isomorphism of $Q$-vector spaces
\[
\QHom(\FF,\FF') \;=\; \Hom_{\P,\Lambda}(\ulM^{(D)}(\FF),\ulM^{(D)}(\FF'))\otimes_{\TA} Q
\]
where the later is the space of all morphisms commuting with $\P$ and $\Lambda(\lambda)$ from (\ref{EQ.Pi+Lambda}) for all $\lambda\in\Ff_{q^l}$.
\end{suchthat}
By lemma~\ref{Lemma2.9a} the condition on the weights can be dropped if $\chr\notin D$ or if $\chr\ne\infty$ and $\infty\in D$.
\end{Proposition}

\begin{proof}
Let $\ulM:=\ulM^{(D)}(\FF)$ and $\ulM':=\ulM^{(D)}(\FF')$. We exhibit a monomorphism of $Q$-vector spaces from $\QHom(\FF,\FF')$ to $\Hom(\ulM,\ulM')\otimes_{\TA}Q$ in case 1 (respectively from $\QHom(\FF,\FF')$ to $\Hom_{\P,\Lambda}(\ulM,\ulM')\otimes_{\TA}Q$ in case 2) and another monomorphism from the target of the first to the $Q$-vector space 
\[
H\;:=\;\bigl\{\,f_{0,\eta}:\F_{0,\eta}\to\F'_{0,\eta}\enspace \text{with}\enspace f_{0,\eta}\circ\P_{0,\eta}^{-1}\circ\t_{0,\eta}\;=\;(\P'_{0,\eta})^{-1}\circ\t'_{0,\eta}\circ\sigma^\ast(f_{0,\eta})\,\bigr\}
\]
introduced in proposition~\ref{PropAltDescrQHom} such that the composition of the two monomorphisms is the isomorphism from \ref{PropAltDescrQHom}.

Let $f\in\QHom(\FF,\FF')$. By the Riemann-Roch Theorem we can find some $a\in Q$ such that $a\cdot f$ maps from $\FF$ into $\FF'(n\cdot D)$ for some $n>0$. Since $a$ and $f$ commute with the $\P$'s and $\t$'s, we get for the first monomorphism
\[
\begin{array}{crccll}
f&\mapsto &a\cdot f_0|_{C_L\setminus D}\otimes a^{-1}&\in&\Hom(\ulM,\ulM')\otimes_A Q&\quad\text{in case 1, and}\\[1mm]
f&\mapsto &a\cdot (f_0\oplus\ldots\oplus f_{l-1})|_{C_L\setminus D}\otimes a^{-1}&\in&\Hom_{\P,\Lambda}(\ulM,\ulM')\otimes_A Q&\quad\text{in case 2.}
\end{array}
\]
\forget{
\[
\begin{array}{rclcrcll}
a\cdot f_0|_{C_L\setminus D}&\in&\Hom(\ulM,\ulM')&\text{ and }& f\mapsto a\cdot f_0|_{C_L\setminus D}\otimes a^{-1}&\in&\Hom(\ulM,\ulM')\otimes_A Q&\quad\text{in case 1, and}\\[1mm]
a\cdot (f_0\oplus\ldots f_{l-1})|_{C_L\setminus D}&\in&\Hom_{\P,\Lambda}(\ulM,\ulM')&\text{ and }& f\mapsto a\cdot (f_0\oplus\ldots f_{l-1})|_{C_L\setminus D}\otimes a^{-1}&\in&\Hom_{\P,\Lambda}(\ulM,\ulM')\otimes_A Q&\quad\text{in case 2}
\end{array}
\]
}
To construct the second monomorphism we treat each case separately.

\smallskip
\noindent
1. The localization $\Hom(\ulM,\ulM')\otimes_{\TA}Q\to H$, $g\otimes a\mapsto ag_\eta$ at the generic point $\eta$ of $C_L$ gives the desired monomorphism.

\smallskip
\noindent
2. Let $(g:\,\bigoplus M_i \rightarrow \bigoplus M_i)\otimes a\in \Hom_{\P,\Lambda}(\ulM,\ulM')\otimes_{\TA}Q$. Then $g$ corresponds to a block matrix \mbox{$(g_{ij}:M_j\to M_i)_{0\le i,j<l}$} with $g_{ij}\cdot\lambda^{q^j} = \lambda^{q^i}\cdot g_{ij}$ for all $\lambda\in \Ff_{q^l}$. Therefore, we have $g_{ij}=0$ for $i\ne j$. We map $g\otimes a$ to the localization $a\cdot(g_{00})_\eta$ at $\eta$. Since $\P g=g\P$ this map is injective and our proof is complete.
\end{proof}

\noindent
{\it Remark.} Again note the importance of the assumption that $\FF$ and $\FF'$ must have the same weight, since otherwise $\QHom(\FF,\FF')=(0)$ by \ref{PROP.1} whereas $\Hom(\ulM^{(D)}(\FF),\ulM^{(D)}(\FF'))$ could be non-zero. Consider for example the abelian $\tau$-sheaves on $C=\PP^1_\Fq$ with $C\setminus\{\infty\}=\Spec \Fq[t]$ given by $\F_i=\O_{\PP^1_L}(i\cdot\infty)$, $\t_i=t$ and $\F'_i=\O_{\PP^1_L}(2i\cdot\infty)$, $\t'_i=t^2$, where $\P$ and $\P'$ are the natural inclusions. They have $\weight(\FF)=1$ and $\weight(\FF')=2$. If we choose $D={\rm V}(t)$ and $z=t^{-1}\in \Gamma(C\setminus D,\O_{\PP^1_\Fq})$ as uniformizing parameter at $\infty$ then $\ulM^{(D)}(\FF)=(L[z],1)=\ulM^{(D)}(\FF')$.

\begin{Definition}\label{Def1.38b}
Let $\ulM$ and $\ulM'$ be pure $A$-motives. Then the elements of $\Hom(\ulM,\ulM')\otimes_A Q$ which admit an inverse in $\Hom(\ulM',\ulM)\otimes_A Q$ are called \emph{quasi-isogenies}.
\end{Definition}

\begin{Corollary}\label{Cor2.9d}
Let the characteristic be different from $\infty$. Then the functor $\FF\mapsto\ulM(\FF)$ defines an equivalence of categories between 
\begin{suchthat}
\item 
the category with abelian $\tau$-sheaves as objects and with $\QHom(\FF,\FF')$ as the set of morphisms,
\item 
 and the category with pure $A$-motives as objects and with $\Hom(\ulM,\ulM')\otimes_A Q$ as the set of morphisms.
\end{suchthat}
We call these the \emph{quasi-isogeny categories} of abelian $\tau$-sheaves of characteristic different from $\infty$ and of pure $A$-motives respectively.
\end{Corollary}

\begin{proof}
This is just a reformulation of Theorem~\ref{PropX.1} and Proposition~\ref{CONNECTION}.
\end{proof}


\bigskip

\section{Simple and semisimple abelian $\tau$-sheaves and pure $A$-motives}

In this section we want to draw some first conclusions about $\QEnd(\FF)$. 

\begin{Definition}
Let $\FF$ be an abelian $\tau$-sheaf.
\begin{suchthat}
\item $\FF$ is called \emph{simple}, if\/ $\FF\not=\ZZ$ and\/ $\FF$ has no abelian quotient $\tau$-sheaves other than $\ZZ$ and\/ $\FF$.
\item $\FF$ is called \emph{semisimple}, if\/ $\FF$ admits, up to quasi-isogeny, a decomposition into a direct sum $\FF\approx\FF_1\oplus\cdots\oplus\FF_n$ of simple abelian $\tau$-sheaves $\FF_j$ $(1\le j\le n)$.
\item $\FF$ is called \emph{primitive}, if its rank and its dimension are relatively prime.
\end{suchthat}
We make the same definition for a pure $A$-motive.
\end{Definition}

\begin{Remark} \label{RemDualBehaviour}
It is not sensible to try defining \emph{simple} abelian $\tau$-sheaves via abelian sub-$\tau$-sheaves, since for example the shifted abelian $\tau$-sheaf $(\F_{i-n},\P_{i-n},\t_{i-n})$ by $n\in\N$ is always a proper abelian sub-$\tau$-sheaf of $(\F_i,\P_i,\t_i)$. 
Furthermore we have for every positive divisor $D$ on $C$ a strict inclusion $\FF(-D)\subset\FF$. This shows that abelian $\tau$-sheaves behave dually to abelian varieties. Namely an abelian variety is called simple if it has no non-trivial abelian subvarieties.
\end{Remark}

\begin{Proposition}\label{Prop1.29a}
Let $\FF$ be an abelian $\tau$-sheaf with characteristic different from $\infty$. Then $\FF$ is \mbox{(semi-)}simple if and only if the pure $A$-motive $\ulM(\FF)$ is (semi-)simple.
\end{Proposition}

\begin{proof}
First let $\FF$ be simple and let $f:\ulM(\FF)\to \ulM'$ be a surjective morphism of pure $A$-motives. By Theorem~\ref{PropX.1} there is an abelian $\tau$-sheaf $\FF'$ with $\ulM'=\ulM(\FF')$. By \ref{CONNECTION} there is an integer $n$ such that $f\in\Hom\bigl(\FF,\FF'(n\cdot\infty)\bigr)$ and $\im f$ is an abelian quotient $\tau$-sheaf of $\FF$ by \ref{IMISABELIAN}. Hence $f$ is injective or $f=0$ proving that $f:\ulM(\FF)\to\ulM'$ is an isomorphism or $\ulM'=0$.

Conversely let $\ulM(\FF)$ be simple and let $f:\FF\to\FF'$ be an abelian quotient $\tau$-sheaf of $\FF$. Then $\ulM(f):\ulM(\FF)\to\ulM(\FF')$ is surjective. So $f=0$ or $f$ is injective proving that $\FF'=\ZZ$ or $f$ is an isomorphism.

Clearly if $\FF$ is semisimple then so is $\ulM(\FF)$. Conversely if $\ulM(\FF)$ is isogenous to a direct sum $\ulM_1\oplus\ldots\oplus\ulM_n$ with $\ulM_i$ simple, then we obtain from \ref{PropX.1} and \ref{Cor2.9b} simple abelian $\tau$-sheaves $\FF_i$ of the same weight with $\ulM(\FF_i)=\ulM_i$ and $\FF\approx\FF_1\oplus\ldots\oplus\FF_n$ by \ref{Cor2.9d}.
\end{proof}

\begin{Proposition}\label{PROP.5}%
Let $\FF$ be an abelian $\tau$-sheaf. If\/ $\FF$ is primitive, then $\FF$ is simple.
\end{Proposition}

\begin{proof}
Let $\TFF$ be an abelian quotient $\tau$-sheaf of $\FF$. Clearly, we have $\Tr\le r$. If $\Tr=0$, then $\TFF=\ZZ$. Otherwise, the surjection $f\in\Hom(\FF,\TFF)$ is non-zero, and by \ref{PROP.1} we get $\Td r = d\Tr$. Since $r$ and $d$ are relatively prime, it follows $\Tr=r$ and $\Td=d$. Therefore, considering the ranks in $\smallexact{0}{}{\ker f_i}{}{\F_i}{f_i}{\TF_i}{}{0}$, we conclude $\ker f_i=0$ and hence $f_i$ is an isomorphism.
\end{proof}

\begin{Corollary}\label{Cor1.50b}
If $\ulM$ is a pure $A$-motive of rank $r$ and dimension $d$ with $(r,d)=1$ then $\ulM$ is simple.\qed
\end{Corollary}

\begin{Proposition}\label{PROP.1.42B}%
Let $\FF$ and $\FF'$ be abelian $\tau$-sheaves of the same rank and dimension. If the characteristic is different from $\infty$ and if\/ $\FF$ is simple, then every non-zero morphism between $\FF$ and\/ $\FF'$ is an isogeny.
\end{Proposition}

\begin{proof}
Let $f\in\Hom(\FF,\FF')$ be a non-zero morphism. Since the characteristic is different from $\infty$, we know by \ref{IMISABELIAN} that $\im f$ is an abelian quotient $\tau$-sheaf. As $\FF$ is simple, we have $\FF\cong\im f$ and therefore $f$ is injective. Thus, by \ref{PROP.1.42A}, $f$ is an isogeny.
\end{proof}

\begin{Remark}\label{Rem1.51b}
Note that the assumption on the characteristic in the proposition and the  theorem below is essential. Namely, the abelian $\tau$-sheaf $\FF$ of Example~\ref{ExTRichter} is primitive, hence simple, but the endomorphism $f$ of $\FF$ constructed there violates the assertions of the proposition and the theorem below.
\end{Remark}

\begin{Theorem}\label{QEND-DIVISION-MATRIX}
Let $\FF$ be an abelian $\tau$-sheaf of characteristic different from $\infty$.
\begin{suchthat}
\item If\/ $\FF$ is simple, then $\QEnd(\FF)$ is a division algebra over $Q$.
\item If\/ $\FF$ is semisimple with decomposition $\FF\approx\FF_1\oplus\cdots\oplus\FF_n$ up to quasi-isogeny into simple abelian $\tau$-sheaves $\FF_j$, then $\QEnd(\FF)$ decomposes into a finite direct sum of full matrix algebras over the division algebras $\QEnd(\FF_j)$ over $Q$.
\end{suchthat}
\end{Theorem}

\noindent {\it Remark.}
We will show in \cite[Theorem~\BHBThmBBB]{BH_B} that over a finite field also the converses to these statements are true.

\begin{proof}
1. We saw in \ref{QEND-Q-ALGEBRA} that $\QEnd(\FF)$ is a $Q$-algebra. By \ref{QISOG-GROUP}, we can invert every quasi-isogeny in $\QIsog(\FF)$. Thus, by proposition \ref{PROP.1.42B}, $\QEnd(\FF)$ is a division algebra.

\smallskip
\noindent
2. Let $\FF\approx\FF_1\oplus\cdots\oplus\FF_n$ be a decomposition into simple abelian $\tau$-sheaves $\FF_j$. By \ref{QHOM-QEND-ISOMORPHIC}, we know that $\QEnd(\FF)\cong\QEnd(\FF_1\oplus\cdots\oplus\FF_n)$, so we just have to consider the decomposition. By proposition \ref{PROP.1.42B}, we only get non-zero morphisms between $\FF_j$ and $\FF_i$, if $\FF_j\approx\FF_i$. Hence we can group the quasi-isogenous $\FF_j$ and decompose $\QEnd(\FF)=E_1\oplus\cdots\oplus E_m$ into their collective endomorphism rings $E_\nu=\QEnd(\FF_{j_{\nu,1}}\oplus\cdots\oplus\FF_{j_{\nu,\mu_\nu}})$, $1\le\nu\le m$, $\sum_{\nu=1}^m \mu_\nu = n$. By \ref{QHOM-QEND-ISOMORPHIC}, every $\QHom(\FF_{j_{\nu,\alpha}},\FF_{j_{\nu,\beta}})$, $1\le\alpha,\beta\le\mu_\nu$, is isomorphic to $\QEnd(\FF_{j_{\nu,1}})$. Hence we conclude that each $E_\nu$ is isomorphic to a matrix algebra over $\QEnd(\FF_{j_{\nu,1}})$ which completes the proof.
\end{proof}

For example if $\FF$ is an abelian $\tau$-sheaf associated to a Drinfeld module, then  $d=1$ and $\FF$ is primitive, hence simple. Also $\chr\ne\infty$ and so $\QEnd(\FF)$ is a division algebra. Together with Corollary~\ref{Cor2.9d} this gives another proof that the endomorphism $Q$-algebra of a Drinfeld module is a division algebra over $Q$.

\vspace{1cm}


\section{The associated local shtukas} \label{SectLS}

Before treating Tate modules in Section~\ref{SectTateModules} we want to attach another local structure to abelian $\tau$-sheaves or pure $A$-motives which is in a sense intermediate on the way to the $v$-adic Tate module, namely the local (iso-)shtuka at $v$. It is the analog of the Dieudonn\'e module of the $p$-divisible group attached to an abelian variety. Note however one fundamental difference. While the Dieudonn\'e module exists only if $p$ equals the characteristic of the base field, there is no such restriction in our theory here. And in fact this would even allow to dispense with Tate modules at all and only work with local (iso-)shtukas. Being not so radical here we shall nevertheless prove the standard facts about Tate modules through the use of local (iso-)shtukas.

To give the definition we introduce the following notation. Let $v\in C$ be a place of $Q$ and let $L\supset\Fq$ be a field. Recall that $A_{v,L}$ denotes the completion of $\O_{C_L}$ at the closed subscheme $v\times\Spec L$ and that $Q_{v,L}=A_{v,L}[\frac{1}{v}]$. Note that $v\times\Spec L$ may consist of more than one point if the intersection of $L$ with the residue field of $v$ is larger than $\Fq$. Then $A_{v,L}$ is not an integral domain and $Q_{v,L}$ is not a field. Local (iso-)shtukas were introduced in \cite{Hl} under the name \emph{Dieudonn\'e $\Fq\dbl z\dbr$-modules} (respectively \emph{Dieudonn\'e $\Fq\dpl z\dpr$-modules}). They are studied in detail in \cite{Anderson2,Kim}; see also \cite{HartlPSp}. Over a field their definition takes the following form.

\begin{Definition}\label{DefLS1}
An \emph{(effective) local $\sigma$-shtuka} at $v$ of rank $r$ over $L$ is a pair $\ulHM=(\hat M,\phi)$ consisting of a free $A_{v,L}$-module $\hat M$ of rank $r$ and an injective $A_{v,L}$-module homomorphism \mbox{$\s\hat M\to \hat M$}.

A \emph{local $\sigma$-isoshtuka} at $v$ of rank $r$ over $L$ is a pair $\ulHN=(\hat N,\phi)$ consisting of a free $Q_{v,L}$-module $\hat N$ of rank $r$ and an isomorphism $\phi:\s\hat N\to\hat N$ of $Q_{v,L}$-modules.

A \emph{morphism} between two local $\sigma$-shtukas $(\hat M,\phi)$ and $(\hat M',\phi')$ at $v$ is an $A_{v,L}$-homomorphism $f:\hat M\to\hat M'$ with $f\phi=\phi'\s(f)$. We denote the set of morphisms from $\ulHM\to\ulHM'$ by $\Hom_{A_{v,L}[\phi]}(\ulHM,\ulHM')$. The similar definition and notation applies to local isoshtukas.
\end{Definition}

\begin{Remark}
Note that so far in the literature \cite{Anderson2,Hl,HartlPSp, Kim, Laumon} it is always assumed that $A_v$ has residue field $\Fq$, the fixed field of $\sigma$ on $L$. So in particular $A_{v,L}$ is an integral domain and $Q_{v,L}$ is a field. For applications to pure $A$-motives this is not a problem since we may reduce to this case by Propositions~\ref{PropLS3} and \ref{PropLS4} below.
\end{Remark}

\begin{Definition}\label{DefLS1b}
A local shtuka $\ulHM=(\hat M,\phi)$ is called \emph{\'etale} if $\phi$ is an isomorphism. The \emph{Tate module} of an \'etale local $\sigma$-shtuka $\ulHM$ at $v$ is the $G:=\Gal(L^\sep/L)$-module of $\phi$-invariants
\[
T_v\ulHM\;:=\;\bigl(\ulHM\otimes_{A_{v,L}}A_{v,L^\sep}\bigr)^\phi\,.
\]
The \emph{rational Tate module} of $\ulHM$ is the $G$-module
\[
V_v\ulHM\;:=\;T_v\ulHM\otimes_{A_v}Q_v\,.
\]
\end{Definition}

It follows from \cite[Proposition~6.1]{TW} that $T_v\ulHM$ is a free $A_v$-module of the same rank than $\ulHM$ and that the natural morphism
\[
T_v\ulHM\otimes_{A_v}A_{v,L^\sep}\isoto \ulHM\otimes_{A_{v,L}}A_{v,L^\sep}
\]
is a $G$- and $\phi$-equivariant isomorphism of $A_{v,L^\sep}$-modules, where on the left module $G$-acts on both factors and $\phi$ is $\id\otimes\s$. Since $(L^\sep)^{G}=L$ we obtain

\begin{Proposition}\label{Prop2.13'}
Let $\ulHM$ and $\ulHM{}'$ be \emph{\'etale} local $\sigma$-shtukas at $v$ over $L$. Then
\begin{suchthat}
\item 
$\DS\ulHM\;=\;(T_v\ulHM\otimes_{A_v}A_{v,L^\sep})^{G}$, the Galois invariants,
\item 
$\DS\Hom_{A_{v,L}[\phi]}(\ulHM,\ulHM{}')\;\isoto\;\Hom_{A_v[G]}(T_v\ulHM,T_v\ulHM{}')\;,\es f\mapsto T_vf$ is an isomorphism.
\end{suchthat}
In particular the Tate module functor yields a fully faithful embedding of the category of \'etale local shtukas at $v$ over $L$ into the category of $A_v[G]$-modules, which are finite free over $A_v$. \qed
\end{Proposition}

\bigskip

If the residue field $\BF_v$ of $v$ is larger than $\Fq$ one has to be a bit careful with local (iso-)shtukas since the ring $Q_{v,L}$ is then in general not a field. Namely let $\#\BF_v=q^n$ and let $\BF_{q^f}:=\{\alpha\in L:\alpha^{q^n}=\alpha\}$ be the ``intersection'' of $\BF_v$ with $L$. Choose and fix an $\BF_q$-homomorphism $\BF_{q^f}\hookrightarrow\BF_v\subset A_v$. Then
\[
\BF_v\otimes_\Fq L = \prod_{\Gal(\BF_{q^f}/\Fq)}\BF_v\otimes_{\BF_{q^f}}L = 
\prod_{i\in\Z/f\Z}\BF_v\otimes_\Fq L\,/\,(b\otimes 1-1\otimes b^{q^i}:b\in \BF_{q^f})
\]
and $\s$ transports the $i$-th factor to the $(i+1)$-th factor. Denote by $\Fa_i$ the ideal of $A_{v,L}$ (or $Q_{v,L}$) generated by $\{b\otimes 1-1\otimes b^{q^i}:b\in \BF_{q^f}\}$. Then
\[
A_{v,L} = \prod_{\Gal(\BF_{q^f}/\Fq)}A_v\wh\otimes_{\BF_{q^f}}L = 
\prod_{i\in\Z/f\Z}A_{v,L}\,/\,\Fa_i
\]
and similarly for $Q_{v,L}$.
Note that the factors in this decomposition and the ideals $\Fa_i$ correspond precisely to the places $v_i$ of $C_{\BF_{q^f}}$ lying above $v$.

\begin{Proposition}\label{PropLS3}
Fix an $i$. The reduction modulo $\Fa_i$ induces equivalences of categories
\begin{suchthat}
\item
$\DS (\hat N,\phi)\longmapsto \bigl(\hat N/\Fa_i\hat N\,,\es\phi^f\mod\Fa_i:(\s)^f \hat N/\Fa_i\hat N \to \hat N/\Fa_i\hat N\bigr)$\\[2mm]
between local $\sigma$-isoshtukas at $v$ over $L$ and local $\sigma^f$-isoshtukas at $v_i$ over $L$ of the same rank.
\item 
\vspace{2mm}
$\DS (\hat M,\phi)\longmapsto \bigl(\hat M/\Fa_i\hat M\,,\es\phi^f\mod\Fa_i:(\s)^f \hat M/\Fa_i\hat M \to \hat M/\Fa_i\hat M\bigr)$\\[2mm]
between\/ \emph{\'etale} local $\sigma$-shtukas at $v$ over $L$ and \emph{\'etale} local $\sigma^f$-shtukas at $v_i$ over $L$ preserving Tate modules
\[
T_v(\hat M,\phi) \isoto T_{v_i}(\hat M/\Fa_i\hat M,\phi^f)\,.
\]
\end{suchthat}
\end{Proposition}

\begin{proof}
Since $\s\Fa_{i-1}=\Fa_{i}$ the isomorphism $\phi$ yields for every $i$ an isomorphism $\phi\mod\Fa_i:\s(\hat N/\Fa_{i-1}\hat N)\to \hat N/\Fa_{i}\hat N$ and similarly for $\hat M$. These allow to reconstruct the other factors from $(\hat N/\Fa_i\hat N,\,\phi^f\mod\Fa_i)$. More precisely we describe the quasi-inverse functor. Let $\ulHN'=(\hat N'\,,\,\psi\!:\!(\s)^f\hat N'\isoto\hat N')$ be a local $\sigma^f$-isoshtuka at $v_i$ over $L$. Define the $Q_{v,L}/\Fa_{i+j}$-module $\hat N_{i+j}:=(\s)^j\hat N'$ for $0\le j<f$ and the $Q_{v,L}/\Fa_{i+j}$-homomorphism 
\[
\begin{array}{rcc@{\!\,}cll}
\phi_{i+j}&:=&\id_{\hat N_{i+j}}&:&\s\hat N_{i+j-1}&\isoto\hat N_{i+j} \qquad\qquad\text{for $0<j< f$ and}\\
\phi_i&:=&\psi&:&\s\hat N_{i+f-1}=(\s)^f\hat N'&\isoto\hat N_i\,. 
\end{array}
\]
The quasi-inverse functor sends $\ulHN'$ to the local $\sigma$-isoshtuka $(\bigoplus_{0\le j<f}\hat N_{i+j}\,,\,\bigoplus_{0\le j<f}\phi_{i+j})$ at $v$ over $L$. Reducing the latter modulo $\Fa_i$ clearly gives back $\ulHN'$. Also note that this quasi-inverse functor sends a morphism $h'$ of local $\sigma^f$-isoshtukas at $v_i$ to the morphism $h:=\bigoplus_{0\le i<f}(\s)^j(h')$ of the corresponding $\sigma$-isoshtukas at $v$.

It remains to show that this functor is indeed quasi-inverse to the reduction modulo $\Fa_i$ functor. For this we need that $\phi\mod\Fa_{i+j}$ is an isomorphism for all $0<j<f$. Namely the required isomorphism
\[
\Bigl(\bigoplus_{0\le j<f}(\s)^j(\hat N/\Fa_i\hat N)\,,\,(\phi^f\mod\Fa_i)\oplus\bigoplus_{0<j<f}\id\Bigr)\;\isoto\;\Bigl(\bigoplus_{0\le j<f}\hat N/\Fa_{i+j}\hat N\,,\,\bigoplus_{0\le j<f}\phi\mod\Fa_{i+j}\Bigr)\;=\;(\hat N,\phi)
\]
is given by $\bigoplus_{0\le j<f}\phi^j\mod\Fa_{i+j}$ and one easily checks that this is a natural transformation. Note that the entry for $j=0$ is $\id_{\hat N/\Fa_i\hat N}$. So we \emph{do not need} that $\phi\mod\Fa_i$ is an isomorphism. Also if $\phi\mod\Fa_i$ is not an isomorphism then $\phi^j\mod\Fa_{i+l}$ is still an isomorphism for $l=j$, but not for $0\le l<j$ which is harmless. This will be crucial in the variant which we prove in Proposition~\ref{PropLS4} below.

For \'etale local shtuka we can use the same argument because there again all Frobenius maps are isomorphism. Finally, the isomorphism between the Tate modules follows from the observation that an element $(x_j)_{j\in\Z/f\Z}$ is $\phi$-invariant if and only if $x_{j+1}=\phi(\s x_j)$ for all $j$ and $x_i=\phi^f((\s)^f x_i)$.
\end{proof}

{\it Remark.} \es The advantage of the (\'etale) local $\sigma^f$-(iso-)shtuka at $v_i$ is that it is a free module over $A_{v,L}/\Fa_i=A_v\wh\otimes_{\BF_{q^f}}L$, and the later ring is an integral domain. So the results from \cite{Anderson2,Hl,HartlPSp, Kim, Laumon} apply.

\bigskip

Now let $\FF$ be an abelian $\tau$-sheaf and $v\in C$ an arbitrary place of $Q$. We define the \emph{local $\sigma$-isoshtuka of $\FF$ at $v$} as 
\[
\ulN_v(\FF)\;:=\;\Bigl(\F_0\otimes_{\O_{C_L}}Q_{v,L}\,,\,\P_0^{-1}\circ\t_0\Bigr)\,.
\]
If $v\ne\infty$ we define the \emph{local $\sigma$-shtuka of $\FF$ at $v$} as
\[
\ulM_v(\FF)\;:=\;\Bigl(\F_0\otimes_{\O_{C_L}}A_{v,L}\,,\,\P_0^{-1}\circ\t_0\Bigr)\,.
\]
Likewise if $\ulM$ is a pure $A$-motive over $L$ and $v\in \Spec A$ we define the \emph{local $\sigma$-(iso-)shtuka of $\ulM$ at $v$} as
\[
\ulM_v(\ulM)\;:=\;\ulM\otimes_{A_L}A_{v,L} \qquad\text{respectively}\qquad \ulN_v(\ulM)\;:=\;\ulM\otimes_{A_L}Q_{v,L}\,.
\]
These local (iso-)shtukas all have rank $r$. The local shtukas are \'etale if $v\ne\chr$. For $v=\infty$ we also define $\ulN_\infty(\ulM)$ in the same way. Note that $\ulN_\infty(\FF)$ and $\ulN_\infty(\ulM)$ do not contain a local $\sigma$-shtuka since they are isoclinic of slope $-\weight(\FF)<0$. 

However, if $v=\infty$ the periodicity condition allows to define a different local (iso-)shtuka at $\infty$ which is of slope $\ge0$. Namely, choose a finite closed subscheme $D\subset C$ as in Section~\ref{SectRelation} with $\infty\notin D$ and set $\TA=\Gamma(C\setminus D,\O_C)$. We define the \emph{big local $\sigma$-shtuka of $\FF$ at $\infty$} as 
\[
\ulTM_\infty(\FF)\;:=\;\ulM^{(D)}(\FF)\otimes_{\TA\otimes_\Fq L}A_{v,L}\;=\;\bigoplus_{i=0}^{l-1}\F_i\otimes_{\O_{C_L}}A_{v,L}
\]
with $\tau$ from (\ref{EQ.Tau}), and the \emph{big local $\sigma$-isoshtuka of $\FF$ at $\infty$} as
\[
\ulTN_\infty(\FF)\;:=\;\ulTM_\infty(\FF)\otimes_{A_{\infty,L}} Q_{\infty,L}\,.
\]
Both have rank $rl$ and depend on the choice of $k,l$ and $z$. If the characteristic is different from $\infty$ then $\ulTM_\infty(\FF)$ is \'etale. Note that $\ulTM_\infty(\FF)$ and $\ulTN_\infty(\FF)$ were used in \cite{Hl} to construct the uniformization at $\infty$ of the moduli spaces of abelian $\tau$-sheaves.

The big local (iso-)shtukas at $\infty$, $\ulTM_\infty(\FF)$ and $\ulTN_\infty(\FF)$ are always equipped with the automorphisms $\P$ and $\Lambda(\lambda)$ for $\lambda\in \Ff_{q^l}\cap L$ from (\ref{EQ.Pi+Lambda}). We let $\Delta_\infty$ be ``the'' central division algebra over $Q_\infty$ of rank $l$ with Hasse invariant $-\frac{k}{l}$, or explicitly
\begin{equation}\label{EqDelta}
\Delta_\infty \,:=\, \Ff_{q^l}\dpl z\dpr\II{\P}\,/\,(\P^l-z^k,\, \lambda z-z\lambda,\, \P\lambda^q-\lambda\P\text{ for all }\lambda\in \Ff_{q^l})\,.
\end{equation}
If $\Ff_{q^l}\subset L$ we identify $\Delta_\infty$ with a subalgebra of $\End_{Q_{\infty,L}[\phi]}\bigl(\ulTN_\infty(\FF)\bigr)$ by mapping $\lambda\in\BF_{q^l}\subset\Delta_\infty$ to $\Lambda(\lambda)$.

\medskip

\forget{
Now let $\FF$ be an abelian $\tau$-sheaf and $v\in C$ an arbitrary place of $Q$. Choose a finite closed subscheme $D\subset C$ as in section~\ref{SectRelation} and set $\TA=\Gamma(C\setminus D,\O_C)$. If $v\ne\infty$ we choose $D=\{\infty\}$, then $\ulM^{(D)}(\FF)$ is the pure $A$-motive associated with $\FF$. We define the \emph{local $\sigma$-shtuka of $\FF$ at $v$} as 
\[
\ulM_v(\FF)\;:=\;\ulM^{(D)}(\FF)\otimes_{\TA\otimes_\Fq L}A_{v,L}
\]
and the \emph{local $\sigma$-isoshtuka of $\FF$ at $v$} as
\[
\ulN_v(\FF)\;:=\;\ulM^{(D)}(\FF)\otimes_{\TA\otimes_\Fq L} Q_{v,L}\,.
\]
Likewise if $\ulM$ is a pure $A$-motive over $L$ and $v\in \Spec A$ we define the \emph{local $\sigma$-(iso-)shtuka of $\ulM$ at $v$} as
\[
\ulM_v(\ulM)\;:=\;\ulM\otimes_{A_L}A_{v,L} \qquad\text{respectively}\qquad \ulN_v(\ulM)\;:=\;\ulM\otimes_{A_L}Q_{v,L}\,.
\]
These local shtukas are \'etale if $v\ne\chr$, even if $v=\infty$. In the later case we will also need another local isoshtuka, namely
\[
\ulTN_\infty(\FF)\;:=\;\bigl(\F_0\otimes_{\O_{C_L}}Q_{\infty,L}\,,\,\P_0^{-1}\circ\t_0\bigr)
\]
for an abelian $\tau$-sheaf $\FF$ and
\[
\ulTN_\infty(\ulM)\;:=\;\ulM\otimes_{A_L}Q_{\infty,L}
\]
for a pure $A$-motive $\ulM$. In contrast to $\ulN_\infty(\FF)$ and $\ulN_\infty(\ulM)$ the later do not contain \'etale local shtuka (except for $\ulTN_\infty(\FF)$ when $\chr=\infty$).

\medskip
}

\begin{Theorem}\label{ThmLS1}
Let $\FF$ and $\FF'$ be abelian $\tau$-sheaves of the same weight over a \emph{finite} field $L$ and let $v$ be an arbitrary place of $Q$. 
\begin{suchthat}
\item 
Then there is a canonical isomorphism of $Q_v$-vector spaces
\[
\QHom(\FF,\FF')\otimes_Q Q_v \isoto \Hom_{Q_{v,L}[\phi]}\bigl(\ulN_v(\FF),\ulN_v(\FF')\bigr)\,.
\]
\item 
If $v=\infty$ choose an $l$ which satisfies condition 2 of \ref{Def1.1} for both $\FF$ and $\FF'$ and assume $\BF_{q^l}\subset L$. Then there is a canonical isomorphism of $Q_\infty$-vector spaces
\[
\QHom(\FF,\FF')\otimes_Q Q_\infty \isoto \Hom_{\Delta_\infty\wh\otimes_\Fq L[\phi]}\bigl(\ulTN_\infty(\FF),\ulTN_\infty(\FF')\bigr)\,.
\]
\end{suchthat}
\end{Theorem}

\begin{proof}
1. Since in the notation of proposition~\ref{PropAltDescrQHom} the condition
\begin{equation}
\label{EqLS1}
f_{0,\eta}\circ\P_{0,\eta}^{-1}\circ\t_{0,\eta}-(\P'_{0,\eta})^{-1}\circ\t'_{0,\eta}\circ\s(f_{0,\eta})\;=\;0  
\end{equation}
is $Q$-linear in $f_{0,\eta}$ we see that the left hand side of the asserted isomorphy is
\[
\{\,f_{0,\eta}:\F_{0,\eta}\otimes_Q Q_v \to \F'_{0,\eta}\otimes_Q Q_v \,|\es f_{0,\eta}\text{ satisfies (\ref{EqLS1})}\,\}\,.
\]
Since $L/\Fq$ is finite, $Q_v\otimes_\Fq L=Q_{v,L}$ and $\F_{0,\eta}\otimes_Q Q_v$ equals $\ulN_v(\FF)$, and 1 is proved.

\smallskip
\noindent
2. Consider the isomorphism 
\[
\Hom\bigl(\ulM^{(D)}(\FF),\ulM^{(D)}(\FF')\bigr)\otimes_{\TA} Q_\infty\es\cong\es\Hom_{Q_{\infty,L}[\phi]}\bigl(\ulTN_\infty(\FF),\ulTN_\infty(\FF')\bigr)
\]
whose existence is proved as in 1. Now 2 follows by applying \ref{CONNECTION} and noting that the commutation with $\P$ and $\Lambda(\lambda)$ cuts out linear subspaces on both sides which become isomorphic.
\end{proof}

\begin{Theorem}\label{ThmLS2}
Let $\ulM$ and $\ulM'$ be pure $A$-motives over a finite field $L$ and let $v\in\Spec A$ be an arbitrary maximal ideal. Then
\[
\Hom(\ulM,\ulM')\otimes_A A_v \isoto \Hom_{A_{v,L}[\phi]}(\ulM_v(\ulM),\ulM_v(\ulM'))\,.
\]
\end{Theorem}

\begin{proof}
The argument of the previous theorem also works here since $A_v$ is flat over $A$.
\end{proof}

\noindent
{\it Remark.}
If one restricts to places $v\ne\chr$, where the local $\sigma$-shtuka is \'etale, Theorems~\ref{ThmLS1} and \ref{ThmLS2} even hold for finitely generated fields. This was shown by Tamagawa~\cite{Tam}; see also Corollary~\ref{Cor2.17'} below.

\medskip

Let now the characteristic be finite and $v=\chr$ be the characteristic point. Consider an abelian $\tau$-sheaf $\FF$ of characteristic $c$, its local $\sigma$-shtuka $\ulM_\chr(\FF)=(\hat M,\phi)$ at $\chr$ and the decomposition of the later described before proposition~\ref{PropLS3}
\[
\ulM_\chr(\FF)=\prod_{i\in\Z/f\Z}\ulM_\chr(\FF)/\Fa_i\ulM_\chr(\FF)\,.
\]
From the morphism $c:\Spec L\to\Spec\BF_\chr\subset C$ we obtain a canonical $\BF_q$-homomorphism $c^\ast:\BF_\chr\hookrightarrow L$, $f=[\BF_\chr:\Fq]$ and a distinguished place $v_0$ of $C_{\BF_\chr}$ above $v=\chr$, namely the image of $c\times c:\Spec L\to C\times\BF_\chr$. Then $\coker\phi$ on $\ulM_\chr(\FF)$ is annihilated by a power of $\Fa_0$ and therefore $\phi$ has no cokernel on $\ulM_\chr(\FF)/\Fa_i\ulM_\chr(\FF)$ for $i\ne0$ and the proof of proposition~\ref{PropLS3} yields

\begin{Proposition}\label{PropLS4}
The reduction modulo $\Fa_0$ 
\[
\ulM_\chr(\FF)\longmapsto \bigl(\ulM_\chr(\FF)/\Fa_0\ulM_\chr(\FF)\,,\,\phi^f\bigr)
\]
induces an equivalence of categories between the category of local
$\sigma$-shtukas at $\chr$ associated with abelian $\tau$-sheaves of
characteristic $c$ and the category of local $\sigma^f$-shtukas at $v_0$ associated with abelian $\tau$-sheaves of characteristic $c$. 
The same is true for pure $A$-motives. \qed
\end{Proposition}

{\it Remark.} \es Now the fixed field of $\sigma^f$ on $L$ equals $\BF_\chr$, the residue field of $A_\chr$. Also $\ulM_\chr(\FF)/\Fa_0\ulM_\chr(\FF)$ is a module over the integral domain $A_\chr\wh\otimes_{\BF_\chr}L$. So again \cite{Anderson2,Hl,HartlPSp, Kim, Laumon} apply to $\bigl(\ulM_\chr(\FF)/\Fa_0\ulM_\chr(\FF),\phi^f\bigr)$.

\forget{
\begin{Proposition}\label{Prop2.18b}
Let $\ulM$ be a pure $A$-motive over $L$ and let $\ulHM{}'_\chr$ be a local $\sigma^f$-subshtuka of $\ulM_\chr(\ulM)/\Fa_0\ulM_\chr(\ulM)$ of the same rank. Then there is a pure $A$-motive $\ulM'$ and an isogeny $f:\ulM'\to\ulM$ with $\ulM_\chr(f)\bigl(\ulM_\chr(\ulM')/\Fa_0\ulM_\chr(\ulM')\bigr)=\ulHM{}'_\chr$. The same is true for abelian $\tau$-sheaves.
\end{Proposition}

\begin{proof}
Extend $\ulHM{}'_\chr$ to the local $\sigma$-subshtuka $\bigoplus_{i\in\Z/f\Z}\phi^i\bigl((\s)^i\ulHM{}'_\chr\bigr)$ of $\ulM_\chr(\ulM)$ and consider 
\[
\ulK\;:=\;\ulM_\chr(\ulM)\,/\,\bigoplus_{i\in\Z/f\Z}\phi^i\bigl((\s)^i\ulHM{}'_\chr\bigr)\,.
\]
The induced morphism $\phi_K:\s K\to K$ has its kernel and cokernel supported on the graph of $c$. Now Proposition~\ref{Prop1.5b} finishes the proof.
\end{proof}

There is a corresponding result at the places $v\ne\chr$ which is stated in Proposition~\ref{Prop2.7b}.
}

\begin{Theorem}\label{ThmW5.2}
For pure $A$-motives over a finite field, being isogenous via a separable isogeny is an equivalence relation.
\end{Theorem}

\begin{proof} 
(cf.~\cite[Theorem 5.2]{Wat})
Since the composition of separable isogenies is again separable we only need to prove symmetry. So let $f:\ulM'\to \ulM$ be a separable isogeny. If the support of $\coker f$ does not meet $\chr$ we can find a dual isogeny which is separable by Remark~\ref{Rem1.26'}. In general we write $\coker f=\ulK^\chr\oplus\ulK_\chr$ with $\Spec(A_L/\chr A_L)\cap\supp\ulK^\chr=\emptyset$ and $\supp\ulK_\chr\subset\Spec(A_L/\chr A_L)$. We factor $f$ as $\ulM'\to\ulM''\to\ulM$ with $\ulM'':=\ker(\ulM\shortonto\coker f\shortonto\ulK_\chr)$ according to Proposition~\ref{Prop1.5b}. By the above we may replace $\ulM'$ by $\ulM''$ and are reduced to the case where $\supp(\coker f)\subset\Spec(A_L/\chr A_L)$. There is a power of $\chr$ which is principal $\chr^n=aA$ for $a\in A$ such that $a$ annihilates $\coker f$. Since our base field is perfect, Lemma~\ref{Lemma1.5d} yields a decomposition
\[
\ulM/a\ulM=(\ulM/a\ulM)^\et\oplus(\ulM/a\ulM)^\nil\,.
\]
We let $\ulM'':=\ker\bigl(\ulM\shortonto(\ulM/a\ulM)^\et\bigr)$ and consider the factorization of $a\cdot\id_\ulM$
\[
\ulM\longto\ulM''\xrightarrow{\es h\;}\ulM'\xrightarrow{\es f\;}\ulM
\]
obtained from the natural surjection $(\ulM/a\ulM)^\et\shortonto \coker f$. Clearly $\coker h$ equals the kernel $\ker\bigl((\ulM/a\ulM)^\et\shortonto\coker f\bigr)$ and $h$ is separable. 

Consider the local $\sigma$-shtuka $\ulM_\chr(\ulM)$ at $\chr$ and the associated local $\sigma^f$-shtuka $\ulM_{v_0}(\ulM):=\ulM_\chr(\ulM)/\Fa_0\ulM_\chr(\ulM)$ from Proposition~\ref{PropLS4}. By \cite[Proposition 2.4.6]{Laumon} the later decomposes $\ulM_{v_0}(\ulM)=\ulM_{v_0}(\ulM)^\et\oplus\ulM_{v_0}(\ulM)^\nil$ into an \'etale part $\ulM_{v_0}(\ulM)^\et$ on which $\sigma^f$ is an isomorphism and a nilpotent part $\ulM_{v_0}(\ulM)^\nil$ on which $\sigma^f$ is topologically nilpotent. Via \ref{PropLS4} we obtain the induced decomposition $\ulM_\chr(\ulM)=\ulM_\chr(\ulM)^\et\oplus\ulM_\chr(\ulM)^\nil$ in which again $\phi$ is an isomorphism on $\ulM_\chr(\ulM)^\et$ and topologically nilpotent on $\ulM_\chr(\ulM)^\nil$. By construction $(\ulM/a\ulM)^\et=\ulM_\chr(\ulM)/\ulM_\chr(\ulM'')$ and
$\ulM_\chr(\ulM'')=a\cdot\ulM_\chr(\ulM)^\et\oplus\ulM_\chr(\ulM)^\nil$. The later is isomorphic to $\ulM_\chr(\ulM)$ as $A_{\chr,L}[\phi]$-module, so $\Hom_{A_{\chr,L}[\phi]}\bigl(\ulM_\chr(\ulM)\,,\,\ulM_\chr(\ulM'')\bigr)$ contains an isomorphism. Since the set of isomorphisms is open we find by Theorem~\ref{ThmLS2} an isogeny $g:\ulM\to\ulM''$ with $\ulM_\chr(g)$ an isomorphism (here we use the assumption that the base field is finite). In particular $g$ is separable and $h\circ g:\ulM\to\ulM'$ is the desired separable isogeny.
\end{proof}

\begin{Example} \label{Ex8.10}
We give an example showing that the preceding theorem is false over infinite fields. This parallels the situation for abelian varieties. Let $C=\PP^1_\Fq$, $A=\Fq[t]$, and $L=\Fq(\gamma)$ where $\gamma$ is transcendental over $\Fq$. Set
\[
T\;:=\;\matr{t+1}{\gamma^{-q}}{\es -\gamma^q}{\es t-1} \qquad\text{and}\qquad T'\;:=\;\matr{t+1-\gamma^{1-q}}{\gamma^{-q}t}{\quad\gamma^q-\gamma}{\quad t}
\]
and consider the pure $A$-motives $\ulM=(L[t]^2,\t=T)$ and $\ulM'=(L[t]^2,\t'=T')$ of characteristic $c^\ast:A\to L, t\mapsto 0$. There is a separable isogeny $f:\ulM'\to\ulM$ given by
\[
\xymatrix @R=0pc
{0 \ar[r] & \ulM' \ar[r]^{\matr{t}{0}{\gamma}{1}} & \ulM \ar[r] &\bigl(L, \t=(1-\gamma^{1-q})\bigr) =\coker f \\
& & {x\choose y} \ar@{|->}[r] & **{ !R(0.45) =<8.5pc,0pc>} \objectbox{(x-\gamma y) \mod t\quad.}
}
\]
We claim that $\End(\ulM')=A=\Fq[t]$. From this it will follow that there is no separable isogeny $g:\ulM\to\ulM'$. Indeed, assume there exists a separable $g$. Then $gf\in \End(\ulM')=A$ is also separable. But $gf$ is not an isomorphism on $\ulM'/t\ulM'$ since already $f$ is not injective modulo $t$. Therefore $gf$ is divisible by $t$, which is not separable. This contradicts the separability of $gf$.

It remains to prove the claim $\End(\ulM')=\Fq[t]$. Let $\sum_{i\ge0}\matr{a_i}{c_i}{b_i}{d_i}t^i\in\End(\ulM')$. The commutation with $\t'$ yields the equations
\begin{eqnarray*}
a_{i-1}+(1-\gamma^{1-q})a_i + \gamma^{-q}b_{i-1} & = & a_{i-1}^q+(1-\gamma^{1-q})a_i^q+(\gamma^q-\gamma)c_i^q\,, \\
(\gamma^q-\gamma)a_i+b_{i-1} & = & b_{i-1}^q+(1-\gamma^{1-q})b_i^q+(\gamma^q-\gamma)d_i^q\,, \\
c_{i-1}+(1-\gamma^{1-q})c_i+\gamma^{-q}d_{i-1} & = & \gamma^{-q}a_{i-1}^q+c_{i-1}^q \,,\\
(\gamma^q-\gamma)c_i+d_{i-1} & = & \gamma^{-q} b_{i-1}^q+d_{i-1}^q\,.
\end{eqnarray*}
For $i=0$ one obtains $c_0=0$ and $a_0\in\Fq$. By subtracting the endomorphism $\matr{a_0}{0}{0}{a_0}$ we may assume that $a_0=0$ and hence $b_0=-\gamma d_0$. When $i=1$ we multiply the first equation by $\gamma^q$ and subtract the second to obtain
\[
b_0^q\;=\;(\gamma^q-\gamma)\bigl(a_1+\gamma c_1-\gamma^{-1}b_1-d_1\bigr)^q\,.
\]
Since $\gamma^q-\gamma$ is not a $q$-th power in $L$ we must have $b_0=d_0=0$ and iterating this argument proves the claim.\qed
\end{Example}


\bigskip

\section{Tate modules}\label{SectTateModules}

In this section we define Tate modules for pure $A$-motives and abelian $\tau$-sheaves and we prove the standard facts on the finiteness of the $A$-module $\Hom(\ulM,\ulM')$ and its relation with Tate modules by using local (iso-)shtukas. We also state the analog of the Tate conjecture for abelian varieties, which was proved by Taguchi~\cite{Taguchi95b} and Tamagawa~\cite{Tam}.

\begin{Definition}
Let $\ulM$ be a\/ $\t$-module on $\TA$ over $L$ (Definition~\ref{Def1.16}) and let\/ $v\in\Spec\TA$ such that the support of\/ $\coker\t$ does not meet\/ $v$. We set
\[
T_v\ulM \;:=\; \liminv{n\in\N}((M/v^nM)\otimes_L\Lsep)^{\textstyle\t} \qquad\text{and}\qquad V_v\ulM \;:=\; T_v\ulM\otimes_\Av\Qv\,,
\]
where the superscript $(\mbox{...})^{\textstyle\t}$ denotes the $\t$-invariants. We call\/ $T_v\ulM$ (respectively $V_v\ulM$) the \emph{(rational) $v$-adic Tate module of $\ulM$}. 
This definition applies in particular if $\ulM$ is a pure $A$-motive.
\end{Definition}

\noindent {\it Remark.} 
Our functor $T_v$ is covariant. In the literature usually the $A_v$-dual of our $T_v \ulM$ is called the $v$-adic Tate module of $\ulM$. With that convention the Tate module functor is contravariant on $\tau$-modules but covariant on Drinfeld modules and Anderson's abelian $t$-modules~\cite{Anderson} (which both give rise to $\tau$-modules). Similarly the classical Tate module functor on abelian varieties is covariant. We chose our non standard convention here solely to avoid perpetual dualizations. This agrees also with the remark that abelian $\tau$-sheaves behave dually to abelian varieties; see~\ref{RemDualBehaviour}.

\smallskip

Next we make similar definitions for abelian $\tau$-sheaves.

\begin{Definition}
Let $\FF$ be an abelian $\tau$-sheaf and let $v\in C$ be a place different from the characteristic point $\chr$. We choose a finite closed subset $D\subset C$ as in section~\ref{SectRelation} with $v\notin D$ and $\infty\in D$ if $v\ne\infty$ and set
\[
T_v\FF \;:=\; T_v\ulM^{(D)}(\FF)\qquad\text{and}\qquad \VvFF \;:=\; V_v\ulM^{(D)}(\FF)\,.
\]
We call\/ $T_v\FF$ (respectively $\VvFF$) the \emph{(rational) $v$-adic Tate module associated to $\FF$}. It is independent of the particular choice of $D$, but if $v=\infty$ it depends on $k,l$ and $z$.
\end{Definition}

By \cite[Proposition~6.1]{TW}, $T_v\FF$ (and $V_v\FF$) are free $A_v$-modules (respectively $Q_v$-vector spaces) of rank $r$ for $v\ne\infty$ and $rl$ for $v=\infty$, which carry a continuous $G=\Gal(L^\sep/L)$-action.

\smallskip

Also the Tate modules $T_\infty\FF$ and $V_\infty\FF$ are always equipped with the automorphisms $\P$ and $\Lambda(\lambda)$ for $\lambda\in \Ff_{q^l}\cap L$ from (\ref{EQ.Pi+Lambda}). 
And if $\Ff_{q^l}\subset L$ we identify the algebra $\Delta_\infty$ from (\ref{EqDelta}) with a subalgebra of $\End_{Q_\infty}(V_\infty\FF)$ by mapping $\lambda\in\BF_{q^l}\subset\Delta_\infty$ to $\Lambda(\lambda)$.

\medskip

The following is evident from the definitions.

\begin{Proposition}\label{PropLS2b}
If $\FF$ is an abelian $\tau$-sheaf over $L$, respectively $\ulM$ a pure $A$-motive over $L$ and $v\in C$ (respectively $v\in \Spec A$) is a place of $Q$ different from the characteristic point $\chr$, then
\begin{eqnarray*}
T_v\FF=T_v(\ulM_v(\FF)) \quad&\text{and}&\quad V_v\FF=V_v(\ulM_v(\FF)) \quad\text{for }v\ne\infty,\\[2mm]
T_\infty\FF=T_\infty(\ulTM_\infty(\FF)) \quad&\text{and}&\quad V_\infty\FF=V_\infty(\ulTM_\infty(\FF)),\\[2mm]
\text{respectively}\quad T_v\ulM=T_v(\ulM_v(\ulM)) \quad&\text{and}&\quad V_v\ulM=V_v(\ulM_v(\ulM))\,.\qed
\end{eqnarray*}
\end{Proposition}

\bigskip

In order to prove the finiteness of $\Hom(\ulM,\ulM')$ we first need the following facts.

\begin{Proposition}\label{PropT.2}
Let $\FF$ and $\FF'$ be abelian $\tau$-sheaves over an arbitrary field $L$ and let $v$ be a place of $Q$ different from $\chr$.
\begin{suchthat}
\item
If $v\ne\infty$ then the natural map is injective
\[
\QHom(\FF,\FF')\otimes_Q Q_v\longto\Hom_{Q_v[G]}(V_v\FF,V_v\FF')\,.
\]
\item 
If $v=\infty$ then the natural map is injective
\[
\QHom(\FF,\FF')\otimes_Q Q_\infty \longto \Hom_{Q_\infty[\P,\Lambda,G]}(V_\infty\FF,V_\infty\FF')\,.
\]
\end{suchthat}
In particular $\QHom(\FF,\FF')$ is a $Q$-vector space of dimension $\le rr'$.
\end{Proposition}

\begin{proof}
1. Consider the morphisms
\[
\QHom(\FF,\FF')\otimes_Q Q_v\into\Hom_{Q_L}(\F_{0,\eta},\F'_{0,\eta})\otimes_{Q_L} (Q_v\otimes_\Fq L) \into\Hom_{Q_{v,L}}\bigl(\ulN_v(\FF),\ulN_v(\FF')\bigr)\,.
\]
obtained from \ref{PropAltDescrQHom} and the definition of $\ulN_v(\FF)$. Clearly the composition factors through $\Hom_{Q_{v,L}[\phi]}\bigl(\ulN_v(\FF),\ulN_v(\FF')\bigr)$. Since in both cases $\ulM_v(\FF)$ and $\ulM_v(\FF')$ are \'etale local shtukas, the isomorphy of the later $Q_v$-vector space with the one stated in the proposition follows by tensoring \ref{Prop2.13'} with $\otimes_{A_v}Q_v$.

\smallskip
\noindent 
2. We adapt the argument from 1 by replacing $\ulN_v$ and $Q_{v,L}$ by $\ulTN_\infty$ and $Q_{\infty,L}[\P,\Lambda]$ and the assertion follows as above.
\end{proof}

The following fact is well known and proved for instance in \cite[Proposition 1.2.4]{Taelman} even without the purity assumption. Nevertheless, we include a proof for the sake of completeness and to illustrate the use of the $\infty$-adic Tate module $V_\infty\FF$ arising from the big local shtuka $\ulTM_\infty(\FF)$.

\begin{Theorem}\label{ThmT.3}
Let $\ulM$ and $\ulM'$ be pure $A$-motives over an arbitrary field $L$. Then $\Hom(\ulM,\ulM')$ is a projective $A$-module of rank $\le rr'$.
\end{Theorem}

\begin{proof}
Since $M'$ is a locally free $A_L$-module, $H:=\Hom(\ulM,\ulM')$ is a torsion free, hence flat $A$-module, because all local rings of $A$ are principal ideal domains. We prove that $H$ is finitely generated by showing that $H$ is a discrete submodule of a finite dimensional $Q_\infty$-vector space. Let $\FF$ and $\FF'$ be abelian $\tau$-sheaves with $\ulM=\ulM(\FF)$ and $\ulM'=\ulM(\FF')$. Then Corollary~\ref{Cor2.9d} and Proposition~\ref{PropT.2} yield inclusions
\[
H\into H\otimes_A Q=\QHom(\FF,\FF')\into\Hom_{Q_\infty}(V_\infty\FF,V_\infty\FF')
\]
The later $Q_\infty$-vector space is finite dimensional and we claim that $H$ is a discrete $A$-lattice. Indeed choose $Q\otimes_\Fq L$-bases $\ul m=(m_1,\ldots,m_{rl})$ of $\bigoplus_{i=0}^{l-1}\F_{i,\eta}$ with $m_j\in M$ and $\ul m'=(m'_1,\ldots,m'_{r'l'})$ of $\bigoplus_{i=0}^{l'-1}\F'_{i,\eta}$ such that $\bigoplus_{i=0}^{l'-1}M'\subset \bigoplus_{j=1}^{r'l'}A_L m'_j$. With respect to these bases every element of $H$ corresponds to a matrix in $M_{r'l'\times rl}(A_L)$. Now choose $Q_\infty$-bases $\ul n$ of $V_\infty\FF$ and $\ul n'$ of $V_\infty\FF'$ and denote the base change matrix from $\ul m$ to $\ul n$ by $B\in\GL_{rl}(Q_{\infty,L^\sep})$, and the base change matrix from $\ul m'$ to $\ul n'$ by $B'\in\GL_{r'l'}(Q_{\infty,L^\sep})$. Then $H$ is contained in
\[
M_{r'l'\times rl}(Q_\infty) \cap B'\cdot M_{r'l'\times rl}(A_L)\cdot B^{-1}
\]
which is discrete in $M_{r'l'\times rl}(Q_\infty)$. This proves that $H$ is a projective $A$-module. The estimate on the rank of $H$ follows from \ref{Cor2.9d} and \ref{PropT.2}.
\end{proof}

\begin{Corollary}\label{CorT.4}
The minimal polynomial of every endomorphism of a pure $A$-motive $\ulM$ lies in $A[x]$. \qed
\end{Corollary}

\begin{Proposition}\label{PropT.1}
Let $\ulM$ and $\ulM'$ be pure $A$-motives over an arbitrary field $L$ and let $v\in\Spec A$ be a maximal ideal different from $\chr$. Then the natural map
\[
\Hom(\ulM,\ulM')\otimes_A A_v \longto \Hom_{A_v[G]}(T_v\ulM,T_v\ulM')
\]
is injective with torsion free cokernel.
\end{Proposition}

\begin{proof}
Consider the morphisms
\[
\Hom(\ulM,\ulM')\otimes_A A_v \into \Hom_{A_L}(M,M')\otimes_{A_L}(A_v\otimes_\Fq L)\into \Hom_{A_L}(M,M')\otimes_{A_L}A_{v,L}
\]
which are injective because $A_v$ is flat over $A$, respectively because $\Hom_{A_L}(M,M')$ is flat over $A_L$. Again the composite morphism factors through
\[
\Hom_{A_{v,L}[\phi]}\bigl(\ulM_v(\ulM),\ulM_v(\ulM')\bigr)\;=\;\Hom_{A_v[G]}(T_v\ulM,T_v\ulM')
\]
(use \ref{Prop2.13'}). To prove that the cokernel is torsion free let $f_v$ be an element of the cokernel which is torsion. Since a power of $v$ is principal we may assume that $g_v=a f_v\in\Hom(\ulM,\ulM')\otimes_A A_v$ for an $a\in A$ with $(a)=v^m$ for some $m\in\N$. Fix a positive integer $n$. There exists a $g\in\Hom(\ulM,\ulM')$ with $g\equiv g_v\mod v^{n+m}$. In particular $a$ divides $g$ in $\Hom_{A_v[G]}(T_v\ulM,T_v\ulM')$. Since
\[
\Bigl((T_v\ulM'/a\cdot T_v\ulM')\otimes_{A/(a)} A_{L^\sep}/(a)\Bigr)^G\;\cong\;\ulM'/a\ulM'
\]
(compare \ref{Prop2.13'}) we see that $g$ maps $\ulM$ into $a\ulM'$. Thus $g$ factors, $g=af$ with $f\in\Hom(\ulM,\ulM')$ and $f\equiv f_v\mod v^n$. As $n$ was arbitrary and $\Hom(\ulM,\ulM')$ is a finitely generated $A$-module the proposition follows.
\end{proof}

If $L$ is finitely generated, Proposition~\ref{PropT.1} was strengthened by Taguchi~\cite{Taguchi95b} and Tamagawa~\cite[\S 2]{Tam} to the following analog of the Tate conjecture for abelian varieties.

\begin{Theorem}[Tate conjecture for $\t$-modules]\label{TATE-CONJECTURE-MODULES}
Let\/ $\ulM$ and $\ulM'$ be two $\t$-modules on $\TA$ over a finitely generated field $L$ and let $G:=\Gal(\Lsep/L)$. Let\/ $v\in\Spec\TA$ such that the support of\/ $\coker\t'$ does not meet\/ $v$. For instance $\ulM$ and $\ulM'$ could be pure $A$-motives, $\TA=A$, and $v\ne\chr$. Then the Tate conjecture holds:
\[
\Hom(\ulM,\ulM')\otimes_{\TA}\Av \;\cong\; \Hom_\AvG(T_v\ulM,T_v\ulM')\,.\qed
\]
\end{Theorem}

\noindent {\it Remark.}
Note one interesting consequence of this result. If $\ulM$ and $\ulM'$ are pure $A$-motives of different weights over a finitely generated field then $\Hom_{A_v[G]}(T_v\ulM,T_v\ulM')=(0)$.

\begin{Theorem}[Tate conjecture for abelian $\tau$-sheaves]\label{TATE-CONJECTURE}
Let $\FF$ and $\FF'$ be abelian $\tau$-sheaves over a finitely generated field $L$ and let $G:=\Gal(\Lsep/L)$. Let\/ $v\in C$ be a place different from the characteristic point $\chr$. 
\begin{suchthat}
\item
If $v\ne\infty$ assume the characteristic $\chr$ is different from $\infty$ or $\weight(\FF)=\weight(\FF')$. Then 
\[
\QHom(\FF,\FF')\otimes_Q\Qv\;\cong\;\Hom_\QvG(\VvFF,\VvFF')\,.
\]
\item
If $v=\infty$ choose an integer $l$ which satisfies condition 2 of \ref{Def1.1} for both $\FF$ and $\FF'$ and assume $\Ff_{q^l}\subset L$. Then 
\[
\QHom(\FF,\FF')\otimes_Q Q_\infty\;\cong\;\Hom_{\Delta_\infty[G]}(V_\infty\FF,V_\infty\FF')\,.
\]
\end{suchthat}
\end{Theorem}

\begin{proof}
1. Set $\ulM:=\ulM(\FF)$ and $\ulM':=\ulM(\FF')$. By \ref{CONNECTION} and \ref{TATE-CONJECTURE-MODULES}, we have
\[
\QHom(\FF,\FF')\otimes_Q Q_v\;\cong\;\Hom(\ulM,\ulM')\otimes_AQ_v \;\cong\;
\Hom_\QvG(V_v\ulM,V_v\ulM')\,.
\]

\smallskip\noindent
2. Let $D\subset C$ be a finite closed subscheme as in Section~\ref{SectRelation} with $\chr,\infty\notin D$ and set $\ulM:=\ulM^{(D)}(\FF)$ and $\ulM':=\ulM^{(D)}(\FF')$. By \ref{CONNECTION} and \ref{TATE-CONJECTURE-MODULES}, we have
\[
\QHom(\FF,\FF')\otimes_Q Q_\infty\;\cong\;\Hom_{\P,\Lambda}(\ulM,\ulM')\otimes_{\TA}Q_\infty \;\cong\;
\Hom_{\Delta_\infty[G]}(V_\infty \ulM,V_\infty \ulM')\,.
\]
Here the last isomorphism comes from the fact that the commutation with $\P$ and $\Lambda(\lambda)$ are linear conditions on $\Hom(\ulM,\ulM')$ and $\Hom(\ulM,\ulM')\otimes_{\TA}Q_\infty\cong\Hom_{Q_\infty[G]}(V_\infty \ulM,V_\infty \ulM')$ thus cutting out isomorphic subspaces.
\end{proof}

As a direct consequence of the theorem together with Proposition~\ref{Prop2.13'} we obtain:

\begin{Corollary}\label{Cor2.17'}
\begin{suchthat}
\item 
Let $\ulM$ and $\ul M'$ be pure $A$-motives over a finitely generated field and let $v\in\Spec A$ be a maximal ideal different from the characteristic point $\chr$, then 
\[
\Hom(\ulM,\ulM')\otimes_A A_v\;\cong\;\Hom_{A_{v,L}[\phi]}\bigl(\ulM_v(\ulM),\ulM_v(\ulM')\bigr)\,.
\]
\item
Let $\FF$ and $\FF'$ be abelian $\tau$-sheaves over a finitely generated field $L$ and let $v$ be a place of $Q$ different from $\chr$ and $\infty$. If $\chr=\infty$ assume $\weight(\FF)=\weight(\FF')$. Then
\[
\QHom(\FF,\FF')\otimes_Q Q_v\;\cong\;\Hom_{Q_{v,L}[\phi]}\bigl(\ulN_v(\FF),\ulN_v(\FF')\bigr)\,.\qed
\]
\end{suchthat}
\end{Corollary}

\bigskip

Finally we establish the relation between Tate modules and isogenies.

\begin{Proposition}\label{Prop2.7b}
\begin{suchthat}
\item 
Let $f:\ulM'\to\ulM$ be an isogeny between pure $A$-motives then $T_vf(T_v\ulM')$ is a $G$-stable lattice in $V_v\ulM$ contained in $T_v\ulM$.
\item 
Conversely if $\ulM$ is a pure $A$-motive and $\Lambda_v$ is a $G$-stable lattice in $V_v\ulM$ contained in $T_v\ulM$, then there exists a pure $A$-motive $\ulM'$ and a separable isogeny $f:\ulM'\to \ulM$ with $T_vf(T_v\ulM')=\Lambda_v$.
\end{suchthat}
\end{Proposition}

\begin{proof}
1 follows from the $G$-invariance of $f$.

\noindent
2. Consider the $A_{v,L^\sep}[G,\phi]$-module $\Lambda_v\otimes_{A_v}A_{v,L^\sep}$. The action of $\phi$ through $\s$ on the right factor and of $G$ through both factors commute. This module is a submodule of 
\[
T_v\ulM\otimes_{A_v}A_{v,L^\sep}\;=\; \ulM_v(\ulM)\otimes_{A_{v,L}}A_{v,L^\sep}
\]
(see Proposition~\ref{Prop2.13'}), and contains $a\cdot \ulM_v(\ulM)\otimes_{A_{v,L}}A_{v,L^\sep}$ for a suitable $a\in A$. Therefore the $G$-invariants $(\Lambda_v\otimes_{A_v}A_{v,L^\sep})^G$ form an \'etale local $\sigma$-subshtuka $\ulHM{}'$ of $\ulM_v(\ulM)$ of the same rank. Now by Proposition~\ref{Prop1.5b} the kernel of the surjection $\ulM\shortonto\ulM_v(\ulM)/\ulHM{}'$ is a pure $A$-motive $\ulM'$ together with a separable isogeny $f:\ulM'\to\ulM$. Clearly $\ulM_v f(\ulM_v\ulM')=\ulHM{}'$ and $T_vf(T_v\ulM')=\Lambda_v$.
\end{proof}

\noindent
{\small
{\it Acknowledgments.} The second author is grateful to the Deutsche Forschungsgemeinschaft for the support in form of DFG-grant HA3006/2-1 and SFB 478. Both authors thank the referee for his careful reading and his useful comments.
}


\vfill

\noindent
\parbox[t]{8cm}{
Matthias Bornhofen  \\ 
Kolleg St. Sebastian\\
Hauptstr. 4\\
D -- 79252 Stegen \\
Germany  \\[0.1cm] }
\parbox[t]{11cm}{ 
Urs Hartl  \\ 
Institute of Mathematics  \\ 
University of Muenster\\
Einsteinstr.\ 62\\
D--48149 Muenster\\
Germany  \\[0.1cm] 
http:/\!/www.math.uni-muenster.de/u/urs.hartl/
}

\end{document}